\documentclass[reqno]{amsart}
\usepackage{amsfonts}
\usepackage{amsmath}
\usepackage{subcaption}
\usepackage{graphicx}
\newtheorem{theorem}{Theorem}

\newtheorem{corollary}[theorem]{Corollary}

\newtheorem{definition}[theorem]{Definition}
\newtheorem{example}[theorem]{Example}

\newtheorem{lemma}[theorem]{Lemma}

\input{tcilatex}
\begin{document}

\title[Theoretical and Numerical Aspect of Fractional Differential Equations]{On Theoretical and Numerical Aspect of Fractional Differential Equations with Purely Integral Conditions}
\author[S. Brahimi, A. Merad and A. Kilicman ]{Saadoune Brahimi$^{1}$, Ahcene Merad$^{2}$ and Adem Kilicman$^{3}$} 
%EndAName
\address{$^{1,2}$ Department of Mathematics, Faculty of Exact Sciences, Larbi Ben
M'hidi University, Oum El Bouaghi, Algeria }
\email{saadoun.brahimi@gmail.com, merad\_ahcene@yahoo.fr, merad.ahcene@univ-oeb.dz}
\address{$^{3}$Department of Mathematics and Institute for Mathematical
	Research, Universiti Putra Malaysia, Serdang 43400 UPM, Selangor, Malaysia}
\email{akilic@upm.edu.my}
\date{\today}
%\maketitle
\thanks{*Corresponding author: akilic@upm.edu.my}

\begin{abstract}
In this paper, we are interested in the study of a problem with fractional
derivatives having boundary conditions of integral types. The problem
represents a Caputo type advection-diffusion equation where the fractional
order derivative with respect to time with $1<\alpha <2$. The method of the
energy inequalities is used to prove the existence and the uniqueness of
solutions of the problem. The finite difference method is also introduced to study
the problem numerically in order to find an approximate solution of the
considered problem. Some numerical examples are presented to show satisfactory results.
\end{abstract}
\maketitle
\section{Introduction}

Fractional Partial Differential Equations (FPDE) are considered as generalizations of
partial differential equations having an arbitrary order and play essential role in engineering, physics and
applied mathematics. Due to the properties of Fractional Differential
Equations $\left( \text{FDE}\right) $, the non-local relationships in space
and time are used to model a complex phenomena, such as in electroanalytical
chemistry, viscoelasticity $[10,21]$, porous environment, fluid
flow, thermodynamic $[11,34,35]$, diffusion transport, rheology $\left[
5,7,15,26,31,33\right] $, electromagnetism, signal processing $[20,21,30]$,
electrical network $[20]$ and others $[9,13,26,27]$. Several problems have
been studied in modern physics and technology by using the partial differential
equations (PDEs) where the nonlocal conditions were described by integrals, further these
integral conditions are of great interest due to their applications in
population dynamics, models of blood circulation, chemical engineering
thermoelasticity $[34]$. At the same time, the existence and uniqueness of the solutions for these type of problems have been studied by several researchers, see for example $[2,12,16,27,28,29]$.
Some results have been obtained by construction of variational formulation and depends on the choice of spaces along their norms, Lax-Milgram theorem,
Poincar\'{e} theorem, fixed point theory. For the numerical studies of (EDPF)
with\textbf{\ classical boundary nonlocal conditions}, we can cite the works
of A. Alikhanov $[3,5,6,7]$, Meerschaert $[15]$, Shen and Liu $%
[26]$ and many others. \\

In this study, we are interested in a problem (FPDE) with \textbf{boundary
conditions of integrals type} $\displaystyle \int_{0}^{1}v\left( x,t\right) dx,$ $\displaystyle
\int_{0}^{1}x^{n}v\left( x,t\right) dx$. For the theoretical study, we use
the energy inequalities method to prove the existence and the uniqueness.
However the numerical study is based on the finite difference method to
obtain an approximate numerical solution of the proposed problem. We use a
uniform discretization of space and time and the fractional operator in the
Caputo sense having order $\alpha $ $(1<\alpha <2)$ is approximated by a scheme
called $L2$ $[26]$, similarly the integer-order differential operators are also 
approximated by central and advanced numerical schemes. For the stability and
convergence of obtained numerical scheme, the conditionally stable method is used and we prove the convergence. Numerical tests are carried out in order to illustrate satisfactory results from the point of view that the values of the approximate solution that is very close to the exact solution. In the process of numerical and
graphical results we applied MATLAB software..

\subsection{Notions and preleminaries}

In this section we recall some early results that we need, such as, the definition of Caputo derivative to explain the problem that
we shall study in this work: let $\Gamma \left( .\right) $ denote the gamma
function. For any positive non-integer value $1<\alpha <2,$ the caputo
derivative defined as follows:

\begin{definition}
$\left( See\text{ }\left[ 12\right] \right) .$ Let us denote by $C_{0}\left(
0,1\right) $ the space of continuous fonctions with compact support in $%
\left( 0,1\right) ,$ and its bilinear form is given by%
\begin{equation}
\left( \left( u,w\right) \right) =\dint\limits_{0}^{1}\Im _{x}^{m}u.\Im
_{x}^{m}wdx\text{ \ \ \ \ \ \ \ \ }\left( m\in \mathbb{N}^{\ast }\right),  \label{66}
\end{equation}
\end{definition}

where%
\begin{equation*}
\Im _{x}^{m}u=\dint\limits_{0}^{x}\frac{\left( x-\xi \right) ^{m-1}}{\left(
m-1\right) !}u\left( \xi ,t\right) d\xi \ \ \ \ \ \ \ \ \left( m\in \mathbb{N}^{\ast }\right) .
\end{equation*}

For $m=1$, we have $\Im _{x}u=\dint\limits_{0}^{x}u\left( \xi ,t\right) d\xi $
and $\Im _{t}u=\dint\limits_{0}^{t}u\left( x,\tau \right) d\tau .$ The
bilinear form $\left( \ref{66}\right) $ is considered as scalar product on $%
C_{0}\left( 0,1\right) $ when is not complete.

\begin{definition}
$\left( See\text{ }\left[ 12\right] \right) .$ We denote \ by 
\begin{equation*}
B_{2}^{m}\left( 0,1\right) =\left\{ 
\begin{array}{c}
L^{2}\left( 0,1\right) \text{ \ \ \ \ \ \ \ \ \ \ \ \ \ \ \ \ }for\text{ }m=0
\\ 
u/\Im _{x}^{m}u\in L^{2}\left( 0,1\right) \text{ \ }\ \text{\ \ \ }for\text{ 
}m\in \mathbb{N}^{\ast },%
\end{array}%
\right.
\end{equation*}%
the completion of $C_{0}\left( 0,1\right) $ for the scalar product defined
by $\left( \ref{66}\right) $.The associated norm to the scalar product is given by %
\begin{equation*}
\left\Vert u\right\Vert _{B_{2}^{m}\left( 0,1\right) }=\left\Vert \Im
_{x}^{m}u\right\Vert _{L^{2}\left( 0,1\right) }=\left(
\dint\limits_{0}^{T}\left( \Im _{x}^{m}u\right) ^{2}dx\right) ^{\frac{1}{2}}.
\end{equation*}
\end{definition}

\begin{lemma}
$\left( See\text{ }\left[ 8\right] \right) .$ For all $m\in \mathbb{N}^{\ast },$ we obtain 
\begin{equation}
\left\Vert u\right\Vert _{B_{2}^{m}\left( 0,1\right) }\leq \left( \frac{1}{2}%
\right) ^{m}\left\Vert u\right\Vert _{L^{2}\left( 0,1\right) }^{2}.
\end{equation}
\end{lemma}

\begin{definition}
$\left( See\text{ }\left[ 12\right] \right) .$ Let $X$ be a Banach space
with the norm $\left\Vert u\right\Vert _{X}$ , and let u :$\left( 0,T\right)
\rightarrow X$ be an abstract functions, by $\left\Vert u\left( .,t\right)
\right\Vert _{X}$ we denote the norm of the element $u\left( .,t\right) \in
X $ at a fixed t.
\end{definition}

We denote by $L^{2}\left( 0,T;X\right) $ the set of all measurable abstract
functions $u\left( .,t\right) $ from $\left( 0,T\right) $ into $X$ such that%
\begin{equation*}
\left\Vert u\right\Vert _{L^{2}\left( 0,T;X\right) }=\left(
\dint\limits_{0}^{T}\left\Vert u\left( .,t\right) \right\Vert _{X}dt\right)
^{\frac{1}{2}}<\infty .
\end{equation*}

\begin{lemma}
$\left[ \text{Cauchy inequality with }\varepsilon \right] \left( See\text{ }%
\left[ 13\right] \right) .$ For all $\varepsilon $ and arbitrary variables
a,b$\in \mathbb{R},$ we have the following inequality:%
\begin{equation}
\left\vert ab\right\vert \leq \frac{\varepsilon }{2}\left\vert a\right\vert
^{2}+\frac{1}{2\varepsilon }\left\vert b\right\vert ^{2}.
\end{equation}
\end{lemma}

\begin{definition}
$\left( See\text{ }\left[ 21\right] \right) .$ The left Caputo derivative
for $1<$ $\alpha <2$ can be expressed as%
\begin{equation*}
_{0}^{c}\partial _{t}^{\alpha }f\left( t\right) =\frac{1}{\Gamma \left(
2-\alpha \right) }\dint\limits_{0}^{t}\frac{f\text{ }"\left( s\right) }{%
\left( t-s\right) ^{\alpha -1}}ds;\text{ }t>0.
\end{equation*}
\end{definition}

\begin{definition}
$\left( See\text{ }\left[ 21\right] \right) .$ The integral of order $\alpha 
$ of the function\ $f\in L^{1}\left[ a,b\right] $ is defined by:%
\begin{equation*}
I_{0}^{\alpha }f\left( t\right) =\frac{1}{\Gamma \left( \alpha \right) }%
\dint\limits_{0}^{t}\frac{f\left( s\right) }{\left( t-s\right) ^{1-\alpha }}%
ds;\text{ }t>0.
\end{equation*}
\end{definition}

\begin{lemma}
$\left( See\text{ }\left[ 1\right] \right) .$ For all real $1<$ $\alpha <2$
we have the inequality%
\begin{equation*}
\int_{0}^{1}\,_{0}^{c}\partial _{t}^{\alpha }\left( \Im _{x}u\right)
^{2}dx\leq 2\int_{0}^{1}\left( _{0}^{c}\partial _{t}^{\alpha }u\right)
\left( \Im _{x}u\right) dx.
\end{equation*}
\end{lemma}

\begin{lemma}
$\left( See\text{ }\left[ 28\right] \right) .$ For all real $1<$ $\alpha <2$
we have the inequality%
\begin{equation*}
\int_{Q}\,\left( _{0}^{c}\partial _{t}^{\alpha }u\right) \left( \Im
_{x}u\right) dxdt\leq \int_{Q}\left( _{0}^{c}\partial _{t}^{\frac{\alpha }{2}%
}\Im _{x}u\right) ^{2}dxdt.
\end{equation*}
\end{lemma}

\section{Statement of the problem}

In the rectangular domain

\begin{equation*}
Q=\left\{ (x,t)\in \mathbb{R}^{2}:0<x<1,\text{ }0<t<T\right\} \text{, \
where }T>0,
\end{equation*}

we consider the fractional differential equation: 
\begin{equation}
\displaystyle \pounds v= \displaystyle _{0}^{c}\partial _{t}^{\alpha }v+a(x,t)\frac{\partial ^{2}v}{%
\partial x^{2}}+b(x,t)\frac{\partial v}{\partial x}+c\left( x,t\right)
v=g\left( x,t\right) ,\text{ where }1<\alpha <2,  \label{1}
\end{equation}

to the equation $(\ref{1})$, we associate the initial conditions: 
\begin{equation}
\left\{ 
\begin{array}{c}
\ell v=v(x,0)=\Phi (x),\qquad \,x\in (0,1), \\ 
\displaystyle qv=\frac{v(x,0)}{\partial t}=\Psi (x),\,\qquad x\in (0,1),%
\end{array}%
\right.  \label{2}
\end{equation}

and the purely integrals conditions 
\begin{equation}
\left\{ 
\begin{array}{c}
\displaystyle \int_{0}^{1}v\left( x,t\right) dx=\mu (t),\qquad t\in (0,T), \\ 
\displaystyle \int_{0}^{1}xv\left( x,t\right) dx=E(t),\qquad t\in (0,T),%
\end{array}%
\right.  \label{3}
\end{equation}

where $\Phi ,\Psi, \mu ,E,a,b,c$ and $g$ are \ known continuous functions.

%\newpage
\textbf{Assumptions:}

\bigskip 1) for all $(x,t)\in \overline{Q}$, we assume that:

\begin{equation}
\underset{Q}{\sup \ }a(x,t)\leq 0,\underset{Q}{\sup }\frac{\partial a^{4}(x,t)%
}{\partial x^{4}}\geq 0,\underset{Q}{\inf }\frac{\partial ^{3}b(x,t)}{%
\partial x^{3}}\leq 0,c(x,t)\geq 0,\underset{Q}{\sup }\frac{\partial
c^{2}(x,t)}{\partial x^{2}}\geq 0,  \label{4a}
\end{equation}

2) for all $(x,t)\in \overline{Q}$, we assume that: 
\begin{eqnarray}
\ 0<M&\leq& 4\frac{\partial ^{2}a(x,t)}{\partial x^{2}}-4\underset{Q}{\sup }%
a(x,t)-\frac{1}{2}\underset{Q}{\sup }\frac{\partial a^{4}(x,t)}{\partial
x^{4}}+\frac{1}{2}\underset{Q}{\inf }\frac{\partial ^{3}b(x,t)}{\partial
x^{3}}\nonumber \\ && \hspace{0.4in} -\frac{1}{2}\underset{Q}{\sup }\frac{\partial c^{2}(x,t)}{\partial
x^{2}}-3\frac{\partial b(x,t)}{\partial x}+2c\left( x,t\right) -\frac{1}{%
2\varepsilon }.  \label{5a}
\end{eqnarray}

3$)$ The functions $\Phi (x)$ and $\Psi (x)$ satisfy the following
compatibility conditions: 
\begin{equation}
\int_{0}^{1}\Phi dx=\mu (0),\,\int_{0}^{1}x\Phi dx=E(0),\int_{0}^{1}\Psi
dx=\mu ^{\prime }(0),\,\int_{0}^{1}x\Psi dx=E^{\prime }(0).  \label{6}
\end{equation}

\noindent We transform a problem $\left( \ref{1}\right)$ -- $\left( \ref{3}\right) $ with
nonhomegenous integral conditions to the equivalent problem with homogenous
integral conditions, by introducing a new unknown function $u$ defined by

\begin{equation}
v(x,t)=\widetilde{u}(x,t)+U(x,t),  \label{7}
\end{equation}%
where%
\begin{equation}
U(x,t)=2(2-3x)\mu (t)+6(2x-1)E(t).  \label{8}
\end{equation}%
Now we study a new problem with homegenous integral conditions

\begin{equation}
\left\{ 
\begin{array}{c}
\displaystyle \pounds \widetilde{u}=_{0}^{c}\partial _{t}^{\alpha }\widetilde{u}+a(x,t)%
\frac{\partial ^{2}\widetilde{u}}{\partial x^{2}}+b(x,t)\frac{\partial 
\widetilde{u}}{\partial x}+c\left( x,t\right) \widetilde{u}=h\left(
x,t\right) , \\ 
\displaystyle \ell v=\widetilde{u}(x,0)=\varphi (x),\qquad x\in (0,1), \\ 
\displaystyle qv=\frac{\widetilde{u}(x,0)}{\partial t}=\psi (x),\qquad x\in (0,1), \\ 
\displaystyle \int_{0}^{1}\widetilde{u}\left( x,t\right) dx=0,\qquad \,t\in (0,T), \\ 
\displaystyle \int_{0}^{1}x\widetilde{u}\left( x,t\right) dx=0,\,\qquad t\in (0,T),%
\end{array}%
\right.  \label{9}
\end{equation}

where%
\begin{eqnarray*}
h(x,t) &=&g(x,t)-\pounds {\normalsize U}(x,t), \\
\varphi (x) &=&\Phi (x)-\ell {\normalsize U}, \\
\psi (x) &=&\Psi (x)-q{\normalsize U}
\end{eqnarray*}

and%
\begin{equation*}
\int_{0}^{1}\varphi (x)dx=0,\,\int_{0}^{1}x\varphi (x)dx=0,\int_{0}^{1}\psi
(x)=0,\,\int_{0}^{1}x\psi (x)=0.
\end{equation*}

Again we introduce new function $u$ defined by 
\begin{equation}
u(x,t)=\widetilde{u}(x,t)-\psi \left( x\right) t-\varphi \left( x\right) ,
\label{10}
\end{equation}

therefore the problem $\left( \ref{9}\right) $ can be given as follow

\begin{equation}
\left\{ 
\begin{array}{c}
\pounds u= \displaystyle _{0}^{c}\partial _{t}^{\alpha }u+a(x,t)\frac{\partial ^{2}u}{%
\partial x^{2}}+b(x,t)\frac{\partial u}{\partial x}+c\left( x,t\right)
u= f\left( x,t\right) , \\ 
\displaystyle \ell u=u(x,0)=0,\,\qquad x\in (0,1), \\ 
\displaystyle qu=\frac{u(x,0)}{\partial t}=0,\qquad \,x\in (0,1), \\ 
\displaystyle \int_{0}^{1}u(x,t)dx=0,\qquad \,t\in (0,T), \\ 
\displaystyle \int_{0}^{1}xu(x,t)dx=0,\qquad \,\,t\in (0,T).%
\end{array}%
\right.  \label{11}
\end{equation}

Thus, instead of seeking a solution $v$ of the problem $\left( \ref{1}%
\right) -\left( \ref{3}\right) $, we establish the existence and 
uniqueness of solution $u$ of the problem $\left( \ref{11}\right) $ and solution $v$ will simply be given by:

\begin{equation}
v(x,t)=\widetilde{u}(x,t)+U(x,t).  \label{12}
\end{equation}
%\end{document}
\section{Inequality of energy and its consequences}

The solution of the problem $\left( \ref{11}\right) $ can be considered as a
solution of the problem in operational form:%
\begin{equation*}
Lu=\tciFourier ,
\end{equation*}
where $L=(\pounds ,\ell ,q)$ is considered from $B$ to $F$, where $B$ is a
Banach space of functions $u\in L{{}^2}(Q)$, whose norm: 
\begin{equation}
\left\Vert u\right\Vert _{B}=\left( \int_{Q}\left( _{0}^{c}\partial _{t}^{%
\frac{\alpha }{2}}\left( \Im _{x}u\right) \right) ^{2}dxdt+\int_{Q}\left(
\Im _{x}u\right) ^{2}dxdt\right) ^{\frac{1}{2}}  \label{13}
\end{equation}

is finite, and $F$ is a Hilbert space consisting of all the elements $F=$ $%
\left( f,0,0\right) $ whose norm is given by:

\begin{equation}
\left\Vert \tciFourier \right\Vert _{F}=\left( \int_{Q}\text{\ }f^{\text{ }%
2}dxdt\right) ^{\frac{1}{2}}.  \label{14}
\end{equation}

Now we let $D(L)$ be the domain of \ the op\'{e}rator $L$ for the
set of all functions $u$ such as that: $\Im _{x}u,$ $\Im _{x}\left(
_{0}^{c}\partial _{t}^{\alpha }u\right) ,$ $\Im _{x}\frac{\partial u}{%
\partial x},$ $\Im _{x}\frac{\partial ^{2}u}{\partial x^{2}}\in L^{2}(Q)$
and $u$ satisfies the integral conditions in problem $\left( \ref{11}\right)
.$ Then,

\begin{theorem}
Under assumptions $\left( \ref{4a}\right) $-$\left( \ref{5a}\right) $,
the condition satisfied then we have the estimate 
\begin{equation}
\left\Vert u\right\Vert _{B}\leq C\left\Vert Lu\right\Vert _{F}\text{,}
\label{15}
\end{equation}%
\ 

where $C$ is a positive constant and independent of $u$ where $u\in D(L)$.
\end{theorem}

\begin{proof}
Multiplying the fractional differential equation in the problem $\left( \ref%
{11}\right) $ by $Mu=-2\Im _{x}^{2}u$ and integrating it on $Q$ we obtain
\begin{eqnarray}
&&-2\int_{Q}\,\left( _{0}^{c}\partial _{t}^{\alpha }u\right) \Im
_{x}^{2}udxdt-2\int_{Q}a(x,t)\frac{\partial ^{2}u}{\partial x^{2}}\Im
_{x}^{2}udxdt  \notag \\
&&-2\int_{Q}b(x,t)\frac{\partial u}{\partial x}\Im
_{x}^{2}udxdt-2\int_{Q}c\left( x,t\right) u\,\Im _{x}^{2}udxdt  \notag \\
&=&-2\int_{Q}f\,\,\Im _{x}^{2}udxdt.  \label{16}
\end{eqnarray}%
Integrating by parts of four integrals in the left side of $\left( \ref%
{16}\right) $, we get

\begin{equation}
-2\int_{Q}\,\left( _{0}^{c}\partial _{t}^{\alpha }u\right) \Im
_{x}^{2}udxdt=2\int_{Q}\left( _{0}^{c}\partial _{t}^{\alpha }\Im
_{x}u\right) \left( \Im _{x}u\right) dxdt,  \label{17}
\end{equation}%
\begin{eqnarray}
-2\int_{Q}a(x,t)\frac{\partial ^{2}u}{\partial x^{2}}\Im _{x}^{2}udxdt
&=&4\int_{Q}\frac{\partial ^{2}a}{\partial x^{2}}\left( \Im _{x}u\right)
^{2}dx-2\int_{Q}au^{2}dxdt  \notag \\
&&-\int_{Q}\frac{\partial a^{4}}{\partial x^{4}}\left( \Im _{x}^{2}u\right)
^{2}dx,  \label{18}
\end{eqnarray}%
\begin{equation}
-2\int_{Q}b(x,t)\frac{\partial u}{\partial x}\Im _{x}^{2}udx=\int_{Q}\frac{%
\partial ^{3}b}{\partial x^{3}}\left( \Im _{x}^{2}u\right) ^{2}dx-3\int_{Q}%
\frac{\partial b}{\partial x}\left( \Im _{x}u\right) ^{2}dx,  \label{19}
\end{equation}%
\begin{equation}
-2\int_{Q}c\left( x,t\right) u\,\Im _{x}^{2}udx=-\int_{Q}\frac{\partial ^{2}c%
}{\partial x^{2}}\left( \Im _{x}^{2}u\right) ^{2}dx+2\int_{Q}c\left( \Im
_{x}u\right) ^{2}dx  \label{20}
\end{equation}%
Substituting $(\ref{17})-(\ref{20})$ in $(\ref{16})$, we have 
\begin{eqnarray}
&&2\int_{Q}\left( _{0}^{c}\partial _{t}^{\alpha }\Im _{x}u\right) \left( \Im
_{x}u\right) dx+4\int_{Q}\frac{\partial ^{2}a}{\partial x^{2}}\left( \Im
_{x}u\right) ^{2}dx-2\int_{Q}au^{2}dx  \notag \\
&&\,-\int_{Q}\frac{\partial a^{4}}{\partial x^{4}}\left( \Im
_{x}^{2}u\right) ^{2}dx+\int_{Q}\frac{\partial ^{3}b}{\partial x^{3}}\left(
\Im _{x}^{2}u\right) ^{2}dx-3\int_{Q}\frac{\partial b}{\partial x}\left( \Im
_{x}u\right) ^{2}dx  \notag \\
&&-\int_{Q}\frac{\partial ^{2}c}{\partial x^{2}}\left( \Im _{x}^{2}u\right)
^{2}dx+2\int_{Q}c\left( \Im _{x}u\right) ^{2}dx  \notag \\
&=&-2\int_{Q}f\,\,\Im _{x}^{2}udx.  \label{21}
\end{eqnarray}%
By the elementary inequalities in lemmas \textbf{(8), (9)} respectively and assumptions $(%
\ref{4a})-(\ref{5a})$ give 
\begin{eqnarray}
&&2\int_{Q}\left( _{0}^{c}\partial _{t}^{\frac{\alpha }{2}}\left( \Im
_{x}u\right) \right) ^{2}dxdt+\int_{Q}(4\frac{\partial ^{2}a}{\partial x^{2}}%
-4\sup au  \notag \\
&&-\frac{1}{2}\frac{\partial a^{4}}{\partial x^{4}}+\frac{1}{2}\inf \frac{%
\partial ^{3}b}{\partial x^{3}}-3\frac{\partial b}{\partial x}  \notag \\
&&-\frac{1}{2}\sup \frac{\partial ^{2}c}{\partial x^{2}}+2c)\left( \Im
_{x}u\right) ^{2}dxdt  \notag \\
&\leq &-2\int_{Q}f\,\,\Im _{x}^{2}udxdt.  \label{22}
\end{eqnarray}%
The estimate of the right side of $(\ref{22})$ gives:%
\begin{eqnarray}
&&\int_{Q}\left( _{0}^{c}\partial _{t}^{\frac{\alpha }{2}}\left( \Im
_{x}u\right) \right) ^{2}dxdt+\int_{Q}(4\frac{\partial ^{2}a}{\partial x^{2}}%
-4\sup au  \notag \\
&&-\frac{1}{2}\frac{\partial a^{4}}{\partial x^{4}}dxdt+\frac{1}{2}\inf 
\frac{\partial ^{3}b}{\partial x^{3}}-3\frac{\partial b}{\partial x}  \notag
\\
&&-2\sup \frac{\partial ^{2}c}{\partial x^{2}}+2c-\frac{1}{2\varepsilon }%
)\left( \Im _{x}u\right) ^{2}dxdt  \notag \\
&\leq &\varepsilon \int_{Q}f^{2}dxdt.  \label{23}
\end{eqnarray}%
So, by using the assumptions $\left( \ref{4a}\right) -\left( \ref{5a}\right) 
$ we find%
\begin{eqnarray}
&&2\int_{Q}\left( _{0}^{c}\partial _{t}^{\frac{\alpha }{2}}\left( \Im
_{x}u\right) \right) ^{2}dxdt+M\int_{Q}\left( \Im _{x}u\right) ^{2}dxdt 
\notag \\
&\leq &\varepsilon \int_{Q}f^{\text{ }2}dxdt  \label{24}
\end{eqnarray}%
Finally, we obtain a priori estimate%
\begin{equation}
\left\Vert u\right\Vert _{B}\leq C\left\Vert Lu\right\Vert _{F}\text{,}
\label{29}
\end{equation}%
where%
\begin{equation*}
C=\left( \frac{\varepsilon }{\min \left( 2,M\right) }\right) ^{\frac{1}{2}}.
\end{equation*}
\end{proof}

\begin{corollary}
A strong solution of problem $\left( \ref{11}\right) $ is unique if it
exists, and depends continuously on $\tciFourier =(f,0,0).$
\end{corollary}

\begin{corollary}
The range of the operator $\overline{L}$ is closed in $F$ and $R(\overline{L}%
)=R(L).$
\end{corollary}

\section{Existence of solutions}

In thei section, we prove the uniqueness of solution, if there is a solution. However, we
have not demonstrated it yet. To do it, we will just prove that $R(L)$ is dense
in $F.$

\begin{theorem}
Let us suppose that the assumptions $\left( \ref{4a}\right) -\left( \ref{5a}%
\right) $and integral conditions $\left( \ref{3}\right) $ are filled, and
for $\omega \in L^{2}(Q)$ and for all $u\in \ D(L)$, we have%
\begin{equation}
\int_{Q}\pounds u.\omega dxdt=0,  \label{30}
\end{equation}%
then $\omega $ almost everywhere in $Q.$
\end{theorem}

\begin{proof}
We can rewrite the equation $\left( \ref{30}\right) $ as follows%
\begin{eqnarray}
\int_{Q}\left( _{0}^{c}\partial _{t}^{\alpha }u\omega \right) dxdt
&=&-\int_{Q}a\left( x,t\right) \frac{\partial ^{2}u}{\partial x^{2}}\omega
dxdt-\int_{Q}b\left( x,t\right) \frac{\partial u}{\partial x}\omega dxdt 
\notag \\
&&-\int_{Q}c\left( x,t\right) u\omega dxdt,  \label{31}
\end{eqnarray}%
Further, we express the function $\omega $ in terms of $u$ as follows :%
\begin{equation}
\omega =-2\Im _{x}^{2}u  \label{32}
\end{equation}%
Substituting $\omega $ by its representation $(\ref{32})$ in $\left( \ref{31}%
\right) ,$ integrating by parts, and taking into account the conditions $\left( %
\ref{3}\right) $, we obtain:%
\begin{equation*}
2\int_{Q}\left( _{0}^{c}\partial _{t}^{\alpha }\Im _{x}u\right) \Im
_{x}udxdt=-4\int_{Q}\frac{\partial ^{2}a}{\partial x^{2}}\left( \Im
_{x}u\right) ^{2}dxdt+2\int_{Q}au^{2}dxdt+\int_{Q}\frac{\partial ^{4}a}{%
\partial x^{4}}\left( \Im _{x}u\right) ^{2}dxdt
\end{equation*}%
\begin{equation*}
-\int_{Q}\frac{\partial ^{3}b}{\partial x^{3}}\left( \Im _{x}u\right)
^{2}dxdt+3\int_{Q}\frac{\partial b}{\partial x}\left( \Im _{x}u\right)
^{2}dxdt+\int_{Q}\frac{\partial ^{2}c}{\partial x^{2}}\left( \Im
_{x}u\right) ^{2}dxdt-2\int_{Q}c\left( \Im _{x}u\right) ^{2}dxdt,
\end{equation*}
\end{proof}
on using under assumptions $\left( \ref{4a}\right) -\left( \ref{5a}\right) $ and
conditions $\left( \ref{6}\right) $, we obtain%
\begin{eqnarray*}
&&2\int_{Q}\left( _{0}^{c}\partial _{t}^{\alpha }\Im _{x}u\right) \Im
_{x}udxdt=-\int_{Q}(4\frac{\partial ^{2}a}{\partial x^{2}}+4\sup au \\
&&+\frac{1}{2}\frac{\partial ^{4}a}{\partial x^{4}}-\frac{1}{2}\inf \frac{%
\partial ^{3}b}{\partial x^{3}}+3\frac{\partial b}{\partial x}+2\sup \frac{%
\partial ^{2}c}{\partial x^{2}}-2c)\left( \Im _{x}u\right) ^{2}dxdt,
\end{eqnarray*}
and this leads that
\begin{equation*}
2\int_{Q}\left( _{0}^{c}\partial _{t}^{\alpha }\Im _{x}u\right) \Im
_{x}udxdt\leq -\left( \frac{1}{2\varepsilon }+M\right) \int_{Q}\left( \Im
_{x}u\right) ^{2}dxdt.
\end{equation*}

By \textbf{lemmas} ($2),(3)$ and ($4)$ we obtain%
\begin{equation*}
2\int_{Q}\left( _{0}^{c}\partial _{t}^{\frac{\alpha }{2}}\left( \Im
_{x}u\right) \right) ^{2}dxdt\leq -\left( \frac{1}{2\varepsilon }+M\right)
\int_{Q}\left( \Im _{x}u\right) ^{2}dxdt.
\end{equation*}

Then%
\begin{equation}
\left( \Im _{x}u\right) ^{2}=0  \label{33}
\end{equation}

and we obtain%
\begin{equation*}
u=0.
\end{equation*}

So $u=0$ in $\Omega $ wich gives $\omega =0$ in $L^{2}(Q).$

\section{\protect\bigskip Finite Difference Method}

\subsection{Discretization of the problem}

Now, we consider a uniform subdivision of intervals $[0,1]$ and $[0,T]$ as follows%
\begin{equation*}
x_{i}=ih;\text{ }i=0,...,N\text{ and }t_{k}=kh_{t};\text{ }k=0,..., M.
\end{equation*}

Then, denote by $v_{i}^{k}$ the approximate solution of $v\left(
x_{i},t_{k}\right) $ at points $(x_{i},t_{k})$, and the operator $L$ is defined
by
\begin{equation}
L=a\frac{\partial ^{2}}{\partial x^{2}}+b\frac{\partial }{\partial x}+c,%
\text{ }L\left( .\right) _{i}^{k}=a_{i}^{k}\frac{\partial ^{2}\left(
.\right) }{\partial x^{2}}+b_{i}^{k}\frac{\partial \left( .\right) }{%
\partial x}+c_{i}^{k}  \label{34}
\end{equation}
where 
\begin{equation*}
a_{i}^{k}=a\left( x_{_{i}},t_{_{k}}\right) ,\quad b_{i}^{k}=b\left(
x_{_{i}},t_{_{k}}\right) ,\quad c_{i}^{k}=c\left( x_{_{i}},t_{_{k}}\right) .
\end{equation*}
From the Taylor devlopment of function $v$ at the point $(x_{i},t_{k})$ we
have%
\begin{equation}
\left( \frac{\partial ^{2}v}{\partial x^{2}}\right) _{i}^{k}=\frac{1}{h^{2}}%
\left( v_{i-1}^{k}-2v_{i}^{k}+v_{i+1}^{k}\right) +O\left( h^{2}\right) ,%
\text{ }\left( \frac{\partial v}{\partial x}\right) _{i}^{k}=\frac{%
v_{i+1}^{k}-v_{i}^{k}}{h}+O\left( h\right).  \label{35}
\end{equation}

Substituting $\left( \ref{35}\right) $ in the operateur $L_{i}^{k}$ expressed
in$\left( \ref{34}\right) $ gives%
\begin{equation}
Lv_{i}^{k+1}=\left( \frac{a_{i}^{k+1}}{h^{2}}+\frac{b_{i}^{k+1}}{h}\right)
v_{i+1}^{k+1}+\left( c_{i}^{k+1}-2\frac{a_{i}^{k+1}}{h^{2}}-\frac{b_{i}^{k+1}%
}{h}\right) v_{i}^{k+1}+\frac{a_{i}^{k+1}}{h^{2}}v_{i-1}^{k+1}.  \label{36}
\end{equation}

The discretization of Caputo derivative fractional operator $%
_{0}^{c}\partial _{t}^{\alpha }v$ $\left[ 17\right]$
with $1<\alpha <2$ defined by 
\begin{equation}
\displaystyle \left( _{0}^{c}\partial _{t}^{\alpha }v\right) _{_{i}}^{k+1}\simeq \gamma
\dsum\limits_{j=0}^{k}\left( v_{_{i}}^{\text{ }k-j-1}-2v_{_{i}}^{\text{ }%
k-j}+v_{_{i}}^{\text{ }k-j+1}\right) d_{j}\text{\ }.  \label{37}
\end{equation}%
\begin{equation*}
\text{where}\left\{ 
\begin{array}{c}
d_{j}=\left( j+1\right) ^{2-\alpha }-j^{2-\alpha } \\ 
d_{0}=1;k=1,...,M\text{ \ }%
\end{array}%
\right. \text{ },\quad \gamma =\frac{h_{t}^{-\text{ }\alpha }}{\Gamma \left(
3-\alpha \right) }.
\end{equation*}

Writing fractional differential equation $\left( \ref{1}\right) $ in points $%
(ih,\left( k+1\right) h_{t})$, we find 
\begin{equation}
\displaystyle \gamma \dsum\limits_{j\text{ }=\text{ }0}^{k}\left( v_{_{i}}^{\text{ }%
k-j-1}-2v_{_{i}}^{\text{ }k-j}+v_{_{i}}^{\text{ }k-j+1}\right)
d_{j}+Lv_{i}^{k+1}=g_{i}^{\text{ }k+1},\text{ }i=\overline{1,N-1}  \label{38}
\end{equation}
then%
\begin{equation}
F_{i}^{k+1}v_{_{i-1}}^{k+1}+A_{i}^{k+1}v_{_{i}}^{k+1}+B_{i}^{k+1}v_{_{i\text{
}+1}}^{k+1}-2\gamma d_{k}v_{_{i}}^{k}+\gamma d_{k}v_{_{i}}^{k-1}+\gamma
\dsum\limits_{j=1}^{k-1}S_{j}d_{j}+\gamma \left(
v_{i}^{-1}-2v_{i}^{0}+v_{i}^{1}\right) d_{k}=g_{i}^{k+1}  \label{39}
\end{equation}
where%
\begin{eqnarray*}
A_{i}^{k+1} &=&\gamma +c_{i}^{k+1}-2\frac{a_{i}^{k+1}}{h^{2}}-\frac{%
b_{i}^{k+1}}{h},\quad B_{i}^{k+1}=\frac{a_{i}^{k+1}}{h^{2}}+\frac{b_{i}^{k+1}%
}{h},\quad \\
F_{i}^{k+1} &=&\frac{a_{i}^{k+1}}{h^{2}},\quad S_{j}=v_{_{i}}^{\text{ }%
k-j-1}-2v_{_{i}}^{\text{ }k-j}+v_{_{i}}^{\text{ }k-j+1}.
\end{eqnarray*}%
In order to eliminate $v_{i}^{-1}$, we use initial condition $\left( \ref{2}\right) $, and we find%
\begin{equation*}
\left( \frac{\partial v}{\partial t}\right) _{i}^{n}\simeq \frac{%
v_{i}^{n}-v_{i}^{n-1}}{h_{t}}
\end{equation*}
therefore%
\begin{equation}
v_{i}^{-1}\simeq \Phi _{i}-h_{t}\Psi _{i}=v_{i}^{0}-h_{t}\Psi _{i}\text{%
,\quad }i=\overline{1,N-1} . \label{40}
\end{equation}

Substituting $\left( \ref{40}\right) $ in $\left( \ref{39}\right) ,$ we
obtain 
\begin{equation}
F_{i}^{k+1}v_{i-1}^{k+1}+A_{i}^{k+1}v_{i}^{k+1}+B_{i}^{k+1}v_{i+1}^{k+1}-2%
\gamma d_{k}v_{i}^{k}+\gamma d_{k}v_{i}^{k-1}+\gamma
\dsum\limits_{j=1}^{k-1}S_{j}d_{j}=d_{k}\gamma v_{i}^{0}+d_{k}\gamma
h_{t}\Psi _{i}-d_{k}\gamma v_{i}^{1}+g_{i}^{k+1}.  \label{41}
\end{equation}

For $k=0$, the relation $\left( \ref{41}\right) $ gives
\begin{equation}
F_{i}^{1}v_{i-1}^{1}+A_{i}^{1}v_{i}^{1}+B_{i}^{1}v_{i+1}^{1}=\gamma
v_{i}^{0}+\gamma h_{t}\Psi _{i}+g_{i}^{1}\text{\quad with\quad }i=\overline{%
1,N-1}.  \label{42}
\end{equation}
By conditions $\left( \ref{3}\right) ,$ and trapezoid method we obtain,
\begin{equation*}
v_{0}^{1}=\frac{2\mu \left( h_{t}\right) -2E\left( h_{t}\right) }{h}%
+2\dsum\limits_{j=1}^{N-1}\left( jh-1\right) v_{j}^{1},\text{ }v_{N}^{1}=%
\frac{2E\left( h_{t}\right) }{h}-2\dsum\limits_{j=1}^{N-1}jhv_{j}^{1}.
\end{equation*}
For $i=1$,%
\begin{equation*}
\left( A_{1}^{1}+2F_{1}^{1}\left( h-1\right) \right) v_{1}^{1}+\left(
B_{1}^{1}+2F_{1}^{1}\left( 2h-1\right) \right)
v_{2}^{1}+2F_{1}^{1}\dsum\limits_{j=3}^{N-1}\left( jh-1\right) v_{j}^{1}
\end{equation*}%
\begin{equation}
=\gamma v_{1}^{0}+\gamma h_{t}\Psi _{1}+g_{1}^{1}+\frac{2F_{1}^{1}}{h}\left(
E\left( h_{t}\right) -\mu \left( h_{t}\right) \right).  \label{43}
\end{equation}
For $i=N-1$,%
\begin{equation*}
-2B_{N-1}^{1}\dsum\limits_{j=1}^{N-3}jhv_{j}^{1}+\left(
F_{N-1}^{1}-2B_{N-1}^{1}\left( N-2\right) h\right) v_{N-2}^{1}+\left(
A_{N-1}^{1}-2B_{N-1}^{1}\left( N-1\right) h\right) v_{N-1}^{1}
\end{equation*}%
\begin{equation}
=\gamma v_{N-1}^{0}+\gamma h_{t}\Psi _{N-1}+g_{N-1}^{1}-\frac{2B_{N-1}^{1}}{h%
}E\left( h_{t}\right).  \label{44}
\end{equation}

\textbf{Matrix's form}

We denote by%
\begin{equation*}
w_{i}=\gamma v_{i}^{0}+\gamma h_{t}\Psi _{i}+g_{i}^{1},\quad y_{1}^{1}=\frac{%
2F_{1}^{1}}{h}\left( E\left( h_{t}\right) -\mu \left( h_{t}\right) \right)
,\quad z_{N-1}^{1}=-\frac{2B_{N-1}^{1}}{h}E\left( h_{t}\right) ,\text{ }
\end{equation*}
\begin{eqnarray*}
P^{1} &=&\left( l_{i,j}\right) _{N-1,N-1}\text{ is square matrix and defined by }
\\
l_{1,1} &=&A_{1}^{1}+2F_{1}^{1}\left( h-1\right) ,\text{ }%
l_{1,2}=B_{1}^{1}+2F_{1}^{1}\left( 2h-1\right) ,\text{ } \\
\text{ }l_{N-1,N-2} &=&F_{N-1}^{1}-2B_{N-1}^{1}\left( N-2\right) h,\text{ }%
l_{N-1,N-1}=A_{N-1}^{1}-2B_{N-1}^{1}\left( N-1\right) h\text{ ,}
\end{eqnarray*}%
\begin{equation*}
l_{i,j}=\left\{ 
\begin{array}{c}
2F_{1}^{1}\left( jh-1\right) \text{\ \ \ when\ \ \ \ \ \ \ \ \ \ }i=1,\text{%
\ }j=\overline{3,N-1}\text{\ \ \ \ \ \ \ \ \ \ \ \ \ \ \ } \\ 
0\text{\ \ \ \ \ \ \ \ \ \ \ \ \ \ \ \ when\ \ \ \ \ \ \ \ \ }\left\vert
i-j\right\vert \geq 2\text{\ ,\ }i=\overline{2,N-2}\text{\ \ \ \ \ \ } \\ 
\text{\ \ }A_{i}^{1}\text{ \ \ \ \ \ \ \ \ \ \ \ \ \ \ when \ \ \ \ \ \ \ \ }%
i=j,\text{\ }i=\overline{2,N-2}\text{\ \ \ \ \ \ \ \ \ \ \ \ \ \ \ \ } \\ 
B_{i}^{1}\text{ \ \ \ \ \ \ \ \ \ \ \ \ \ \ when\ \ \ \ \ \ \ \ \ }i=j-1,%
\text{ }i=\overline{1,N-2}\text{ \ \ \ \ \ \ \ \ } \\ 
\text{\ \ \ }F_{i}^{1}\text{ \ \ \ \ \ \ \ \ \ \ \ \ \ \ when\ \ \ \ \ \ \ \
\ }i=j+1,\text{ }i=\overline{2,N-1}\text{ \ \ \ \ \ \ \ \ \ \ \ } \\ 
-2B_{N-1}^{1}jh\text{ \ \ \ \ \ \ \ when\ \ \ \ \ \ \ \ \ }i=N-1,\text{\ }j=%
\overline{1,N-3}.\text{ \ \ \ \ \ \ \ \ \ \ }%
\end{array}%
\right.
\end{equation*}
Taking account $\left( \ref{42}\right) ,\left( \ref{43}\right) ,$ and $%
\left( \ref{44}\right) ,$ we obtain the matrix system%
\begin{equation}
P^{1}.V^{1}=H^{1}  \label{45}
\end{equation}
where%
\begin{equation*}
H^{1}=W\text{ }^{1}+R^{1},\text{ }W\text{ }^{1}=\left(
w_{1}^{1},w_{2}^{1},...,w_{N-1}^{1}\right) ^{T},\text{ }R^{1}=\left(
y_{1}^{1},0,0,...,0,z_{N-1}^{1}\right) ^{T}.\text{ }
\end{equation*}
To solve the system $\left( \ref{45}\right) $ we can apply one of direct
methods.

\subsection{\textbf{General case }}
It is readily checked that, for $k\geq 1$%
\begin{equation}
\dsum%
\limits_{j=1}^{k-1}S_{j}d_{j}=(d_{2}-2d_{1})v_{i}^{k-1}+d_{1}v_{i}^{k}+d_{k-1}v_{i}^{0}+\left( d_{k-2}-2d_{k-1}\right) v_{i}^{1}+\dsum\limits_{m=2}^{k-2}\sigma _{m}v_{i}^{k-m}
\label{46}
\end{equation}%
\begin{equation}
\text{where \ \ \ }\sigma _{m}=d_{m-1}-2d_{m}+d_{m+1},\text{ }m=\overline{%
2,k-2}.  \notag
\end{equation}%

\begin{lemma}
If $\ k\geq 1;$ \textit{the numerical scheme }$\left( \ref{41}\right) $ is
equivalent to%
\begin{equation*}
F_{i}^{k+1}v_{_{i-1}}^{k+1}+A_{i}^{k+1}v_{_{i}}^{k+1}+B_{i}^{k+1}v_{_{i+1}}^{k+1}=-\gamma \dsum\limits_{m=1}^{k-1}\sigma _{m}v_{i}^{k-m}+\gamma \left( 2-d_{1}\right) v_{_{i}}^{k}+\gamma \left( d_{k}-d_{k-1}\right) v_{_{i}}^{0}
\end{equation*}%
\begin{equation}
+\gamma d_{k}h_{t}\Psi _{i}+g_{i}^{k+1},\quad \text{for \quad }i=1,\ldots
,N-1  \label{47}
\end{equation}
\end{lemma}

\begin{proof}
From the scheme $\left( \ref{41}\right) $, we have%
\begin{equation*}
F_{i}^{k+1}v_{i-1}^{k+1}+A_{i}^{k+1}v_{i}^{k+1}+B_{i}^{k+1}v_{i+1}^{k+1}-2%
\gamma d_{k}v_{i}^{k}+\gamma d_{k}v_{i}^{k-1}+\gamma
\dsum\limits_{j=1}^{k-1}S_{j}d_{j}=d_{k}\gamma v_{i}^{0}+d_{k}\gamma
h_{t}\Psi _{i}-d_{k}\gamma v_{i}^{1}+g_{i}^{k+1}
\end{equation*}
\end{proof}
so%
\begin{equation}
F_{i}^{k+1}v_{i-1}^{k+1}+A_{i}^{k+1}v_{i}^{k+1}+B_{i}^{k+1}v_{i+1}^{k+1}+%
\gamma \dsum\limits_{j=2}^{k-2}S_{j}d_{i}+\gamma
(v_{i}^{k+1}-2v_{i}^{k}+v_{i}^{k-1})\text{ }d_{0}+\gamma
(v_{i}^{1}-2v_{i}^{0}+v_{i}^{-1})d_{k}=g_{i}^{k+1}  \label{48}
\end{equation}
using $\left( \ref{46}\right) $ we obtain%
\begin{equation}
F_{i}^{k+1}v_{i-1}^{k+1}+A_{i}^{k+1}v_{i}^{k+1}+B_{i}^{k+1}v_{i+1}^{k+1}=-%
\gamma \dsum\limits_{m=1}^{k-1}\sigma _{m}v_{i}^{k-m}+\gamma \left(
2-d_{1}\right) v_{i}^{k}+\gamma \left( d_{k}-d_{k-1}\right) v_{i}^{0}  \notag
\end{equation}

\begin{equation}
+\gamma d_{k}h_{t}\Psi _{i}+g_{i}^{k+1},\text{ for\quad }i=1,\ldots ,N-1
\label{49}
\end{equation}
Using the conditions $\left( \ref{3}\right) ,$ and by trapezoid method we obtain: For $i=1$,%
\begin{equation*}
\left( A_{1}^{k+1}+2F_{1}^{k+1}\left( h-1\right) \right) v_{1}^{k+1}+\left(
B_{1}^{k+1}+2F_{1}^{k+1}\left( 2h-1\right) \right)
v_{2}^{k+1}+2F_{1}^{k+1}\dsum\limits_{j=3}^{N-1}\left( jh-1\right)
v_{j}^{k+1}
\end{equation*}

\begin{equation*}
=\frac{2F_{1}^{k+1}}{h}\left( E\left( \left( k+1\right) h_{t}\right) -\mu
\left( \left( k+1\right) h_{t}\right) \right) -\gamma
\dsum\limits_{m=1}^{k-1}\sigma _{m}v_{1}^{k-m}+\gamma \left(
d_{k}-d_{k-1}\right) v_{1}^{0}+\gamma d_{k}h_{t}\Psi _{1}+g_{1}^{k+1}.
\end{equation*}

For $i=N-1$,%
\begin{equation*}
-2B_{N-1}^{k+1}\dsum\limits_{j=1}^{N-3}jhv_{j}^{k+1}+\left(
F_{N-1}^{k+1}-2B_{N-1}^{k+1}\left( N-2\right) h\right) v_{N-2}^{k+1}+\left(
A_{N-1}^{k+1}-2B_{N-1}^{k+1}\left( N-1\right) h\right) v_{N-1}^{k+1}
\end{equation*}%
\begin{equation}
=-\frac{2B_{N-1}^{k+1}}{h}E\left( \left( k+1\right) h_{t}\right) -\gamma
\dsum\limits_{m=1}^{k-1}\sigma _{m}v_{N-1}^{k-m}+\gamma \left(
2-d_{1}\right) v_{N-1}^{k}+\gamma \left( d_{k}-d_{k-1}\right)
v_{N-1}^{0}+\gamma d_{k}h_{t}\Psi _{N-1}+g_{N-1}^{k+1}.  \label{51}
\end{equation}

\textbf{Matrix's form}

We take expression $\left( \ref{49}\right) $ for $i=\overline{2,N-2}$ and
equations $\left( \ref{50}\right) $, $\left( \ref{51}\right) $ to formulate
the matrix systems:%
\begin{equation}
\left\{ 
\begin{array}{c}
P^{k+1}V^{k+1}=H^{k+1}\text{;}\quad k\geq 1 \\ 
\\ 
V^{0},\text{ }V^{1}\text{ are known}%
\end{array}%
\right.  \label{52}
\end{equation}
where%
\begin{eqnarray*}
P^{k+1} &=&\left( l_{i,j}^{k+1}\right) _{N-1,N-1}\text{ is square matrix
defined by } \\
l_{1,1}^{k+1} &=&A_{1}^{k+1}+2F_{1}^{k+1}\left( h-1\right) ,\text{ }%
l_{1,2}^{k+1}=B_{1}^{k+1}+2F_{1}^{k+1}\left( 2h-1\right) ,\text{ } \\
\text{ }l_{N-1,N-2}^{k+1} &=&F_{N-1}^{k+1}-2B_{N-1}^{k+1}\left( N-2\right) h,%
\text{ }l_{N-1,N-1}^{k+1}=A_{N-1}^{k+1}-2B_{N-1}^{k+1}\left( N-1\right) h%
\text{ ,}
\end{eqnarray*}%
\begin{equation*}
l_{i,j}^{k+1}=\left\{ 
\begin{array}{c}
2F_{1}^{k+1}\left( jh-1\right) \text{\ \ \ when\ \ \ \ \ \ \ }i=1,\text{\ }j=%
\overline{3,N-1}\text{ \ \ \ \ \ \ \ \ \ \ \ \ \ \ \ } \\ 
0\text{\ \ \ \ \ \ \ \ \ \ \ \ \ \ \ \ \ \ \ \ when\ \ \ \ \ \ \ }\left\vert
i-j\right\vert \geq 2\text{\ ,\ }i=\overline{2,N-2}\text{\ \ \ \ \ \ \ } \\ 
A_{i}^{k+1}\text{ \ \ \ \ \ \ \ \ \ \ \ \ \ \ when\ \ \ \ \ \ }i=j,\text{\ }%
i=\overline{2,N-2}\text{\ \ \ \ \ \ \ \ \ \ \ \ \ \ \ \ } \\ 
\text{\ }B_{i}^{k+1}\text{ \ \ \ \ \ \ \ \ \ \ \ \ \ \ when \ \ \ \ \ }i=j-1,%
\text{ }i=\overline{1,N-2}\text{\ \ \ \ \ \ \ \ \ \ \ } \\ 
\text{ \ }F_{i}^{k+1}\text{ \ \ \ \ \ \ \ \ \ \ \ \ \ when\ \ \ \ \ \ }i=j+1,%
\text{ }i=\overline{2,N-1}\text{ \ \ \ \ \ \ \ \ \ \ } \\ 
-2B_{N-1}^{k+1}jh\text{ \ \ \ \ \ \ \ \ \ when\ \ \ \ \ \ }i=N-1,\text{\ }j=%
\overline{1,N-3}\text{\ \ \ \ \ \ \ \ \ \ }%
\end{array}%
\right.
\end{equation*}
and%
\begin{eqnarray*}
V\text{ }^{k+1} &=&\left( v_{1}^{k+1},...,v_{N-1}^{k+1}\right) ^{T};\text{ }%
H^{k+1}=-\gamma \dsum\limits_{m=1}^{k-1}\sigma _{m}V^{\text{ }k-m}+W\text{ }%
^{k+1}+R\text{ }^{k+1}\text{; }k\geq 1 \\
W\text{ }^{k+1} &=&\left( w_{1}^{k+1},w_{2}^{k+1},...,w_{N-1}^{k+1}\right)
^{T},\text{ }R^{k+1}=\left( y_{1}^{k+1},0,0,...,0,z_{N-1}^{k+1}\right) ^{T},
\\
w_{i}^{k+1} &=&\gamma \left( 2-d_{1}\right) v_{i}^{k}+\gamma \left(
d_{k}-2d_{k-1}\right) v_{i}^{0}+\gamma d_{k}h_{t}\Psi _{i}+g_{i}^{k+1},\text{
} \\
y_{1}^{k+1} &=&\frac{2F_{1}^{k+1}}{h}\left( E\left( \left( k+1\right)
h_{t}\right) -\mu \left( \left( k+1\right) h_{t}\right) \right) ;\text{ }%
z_{N-1}^{k+1}=-\frac{2B_{N-1}^{k+1}}{h}E\left( \left( k+1\right)
h_{t}\right) .
\end{eqnarray*}

In order to prove system $\left( \ref{52}\right) $ has a unique solution we
denote $\rho $ as an eigenvalue of the matrix $P^{k}$, and $X=\left(
x_{1},x_{2},...,x_{N-1}\right) ^{T}$ is an nonzero eigenvector corresponding
to $\rho $. Then, we choose $i$ such as 
\begin{equation*}
\left\vert x_{i}\right\vert =\max\{|x_{j}|:j=1;...;N-1\}
\end{equation*}
then%
\begin{equation*}
\text{ }\sum_{j=1}^{N-1}l_{i,j}x_{j}=\rho x_{i};\text{ }i=\overline{1;N-1}%
\text{ }
\end{equation*}
therefore%
\begin{equation}
\rho =l_{i,i}+\sum_{\substack{ j=1  \\ j\neq i}}^{N-1}l_{i,\text{ }j}\frac{%
x_{j}}{x_{i}}.  \label{53}
\end{equation}

Substituting the values of $l_{i,j}$ into $\left( \ref{53}\right) ,$ and
taking into account that $F_{i}^{k}$, $a_{i}^{k}$ are negative and $%
\left\vert \frac{x_{j}}{x_{i}}\right\vert \leq 1$ we get, for $i=1$,%
\begin{eqnarray*}
\rho &=&\left( A_{1}^{k+1}+2F_{1}^{k+1}\left( h-1\right) \right) +\left(
B_{1}^{k+1}+2F_{1}^{k+1}\left( 2h-1\right) \right) \frac{x_{2}}{x_{1}}%
+2F_{1}^{k+1}\dsum\limits_{j=3}^{N-1}\left( jh-1\right) \frac{x_{j}}{x_{1}}
\\
&=&\gamma +c_{i}^{k+1}-F_{1}^{k+1}-B_{1}^{k+1}\left( 1-\frac{x_{2}}{x_{1}}%
\right) +2F_{1}^{k+1}\dsum\limits_{j=2}^{N-1}\left( jh-1\right) \frac{x_{j}}{%
x_{1}}.
\end{eqnarray*}

For $i=N-1$,%
\begin{eqnarray*}
\rho &=&l_{i,i}+\sum_{\substack{ j=1  \\ j\neq i}}^{N-1}l_{i,\text{ }j}\frac{%
x_{j}}{x_{i}} \\
&=&A_{N-1}^{k+1}-2B_{N-1}^{k+1}\left( N-1\right) h+\left(
F_{N-1}^{k+1}-2B_{N-1}^{k+1}\left( N-2\right) h\right) \left( \frac{x_{N-2}}{%
x_{N-1}}\right) -2B_{N-1}^{k+1}\dsum\limits_{j=1}^{N-3}jh\frac{x_{j}}{x_{N-1}%
} \\
&=&\gamma +c_{N-1}^{k+1}-B_{N-1}^{k+1}+F_{N-1}^{k+1}\left( \frac{x_{N-2}}{%
x_{N-1}}-1\right) -2B_{N-1}^{k+1}\left( N-1\right)
h-2B_{N-1}^{k+1}\dsum\limits_{j=1}^{N-2}jh\frac{x_{j}}{x_{N-1}}.
\end{eqnarray*}%
For $i=\overline{2,N-2}$,%
\begin{eqnarray}
\rho &=&l_{i,i}+\sum_{\substack{ j=1  \\ j\neq i}}^{N-1}l_{i,j}\frac{x_{j}}{%
x_{i}}  \notag \\
&=&A_{i}^{k+1}+F_{i}^{k+1}\frac{x_{i-1}}{x_{i}}+B_{i}^{k+1}\frac{x_{i+1}}{%
x_{i}}  \notag \\
&=&\gamma +c_{i}^{k+1}-B_{i}^{k+1}-F_{i}^{k+1}+F_{i}^{k+1}\frac{x_{i-1}}{%
x_{i}}+B_{i}^{k+1}\frac{x_{i+1}}{x_{i}}  \notag \\
&=&\gamma +c_{i}^{k+1}+F_{i}^{k+1}\left( \frac{x_{i-1}}{x_{i}}-1\right) +%
\frac{a_{i}^{k+1}+hb_{i}^{k+1}}{h^{2}}\left( \frac{x_{i+1}}{x_{i}}-1\right) .
\end{eqnarray}
From the above we conclude for $i=\overline{1,N-1},$ if $b_{i}^{k+1}<0,$ $B_{N-1}^{k+1}<0$ then $\rho >0.$ If $b_{i}^{k+1}>0$ and $\displaystyle h\leq $ $\min_{1\leq i\leq N-1}\left( \frac{%
-a_{i}^{k+1}}{b_{i}^{k+1}}\right) ,$ $\rho >0,$ then all eigenvalues of
matrix $P^{k+1}$ are strictly positive, therefore $P^{k+1}$ is invertible.

\subsection{Stability}

Since, we have%
\begin{equation*}
F_{i}^{k+1}+A_{i}^{k+1}+B_{i}^{k+1}=\gamma +c_{i}^{k+1},\text{ }%
F_{i}^{k+1}\leq 0,\text{ }A_{i}^{k+1}+B_{i}^{k+1}\geq 0, 
\end{equation*}
then we let $u_{i}^{k+1}$ be the approximate solution of $\left( \ref{49}\right) ,$
and $e_{i}^{k+1}$, the error at point $\left( x_{i},t_{k+1}\right) $ defined
by 
\begin{equation*}
v_{i}^{k+1}-u_{i}^{k+1}=e_{i}^{k+1},\text{ and }\left\Vert E^{\text{ }%
k}\right\Vert =\limfunc{Max}_{\text{ }1\text{ }\leq \text{ }i\text{ }\leq 
\text{ }N-1}|e_{i}^{k}|,\text{ }E^{\text{ }k}=\left(
e_{1}^{k},...,e_{N-1}^{k}\right) ^{T},
\end{equation*}
for $k=0$ \ we apply $\left( \ref{42}\right) $ we get%
\begin{eqnarray*}
\left\Vert E^{1}\right\Vert &\leq &\left( \gamma +c_{i}^{1}\right)
\left\Vert E^{1}\right\Vert =\left( F_{i}^{1}+A_{i}^{1}+B_{i}^{1}\right)
\left\Vert E^{1}\right\Vert \\
&=&\left( F_{i}^{1}\left\Vert E^{1}\right\Vert +\left(
A_{i}^{1}+B_{i}^{1}\right) \left\Vert E^{1}\right\Vert \right) \\
&\leq &\left( \left( A_{i}^{1}+B_{i}^{1}\right) \left\Vert E^{1}\right\Vert
+F_{i}^{1}\left\vert e_{i-1}^{1}\right\vert \right) \\
&\leq &\limfunc{Max}_{\text{ }1\text{ }\leq \text{ }i\text{ }\leq \text{ }%
N-1}\left\vert
F_{i}^{1}e_{i-1}^{1}+A_{i}^{1}e_{i}^{1}+B_{i}^{1}e_{i+1}^{1}\right\vert
=\gamma \left\Vert E^{0}\right\Vert
\end{eqnarray*}
so%
\begin{equation}
\left\Vert E^{1}\right\Vert \preceq \frac{\gamma }{\gamma +c_{i}^{1}}%
\left\Vert E^{0}\right\Vert \preceq \left\Vert E^{0}\right\Vert .  \label{57}
\end{equation}
Therefore the method is stable.

\begin{lemma}
For $k\geq 1$ the scheme $\left( \ref{48}\right) $ is stable and we have%
\begin{equation*}
\left\Vert E^{k+1}\right\Vert \leq C\left\Vert E^{0}\right\Vert ,\text{ }C>0%
\text{, for all }k\geq 1
\end{equation*}
\end{lemma}

\begin{proof}
We use the mathematical induction.
\end{proof}
We assume $\left\Vert E^{j}\right\Vert \leq c_{j}\left\Vert E^{0}\right\Vert,$ and $C_{\max }=\max c_{j};$ \ where $c_{j}\succ 0,$ $j=\overline{1,k}$
from $\left( \ref{49}\right) $ we get%
\begin{equation}
F_{i}^{k+1}e_{i-1}^{k+1}+A_{i}^{k+1}e_{i}^{k+1}+B_{i}^{k+1}e_{i+1}^{k+1}=-%
\gamma \dsum\limits_{m=1}^{k-1}\sigma _{m}e_{i}^{k-m}+\gamma \left(
2-d_{1}\right) e_{_{i}}^{k}+\gamma \left( d_{k}-d_{k-1}\right) e_{i}^{0},%
\text{ }i=\overline{1,N-1},  \notag
\end{equation}
so%
\begin{eqnarray*}
\left( \gamma +c_{i}^{k+1}\right) \left\Vert E^{k+1}\right\Vert &\leq
&\left( \left( A_{i}^{k+1}+B_{i}^{k+1}\right) \left\Vert E^{k+1}\right\Vert
+F_{i}^{k+1}\left\vert e_{i-1}^{k+1}\right\vert \right) \\
&\leq &\left\Vert -\gamma \dsum\limits_{m=1}^{k-1}\sigma
_{m}e_{i}^{k-m}+\gamma \left( 2-d_{1}\right) e_{_{i}}^{k}+\gamma \left(
d_{k}-d_{k-1}\right) e_{i}^{0}\right\Vert \\
&\leq &\gamma \left( \dsum\limits_{m=1}^{k-1}\left\vert \sigma
_{m}\right\vert \left\Vert E_{i}^{k-m}\right\Vert +\left( 2-d_{1}\right)
\left\Vert e_{i}^{k}\right\Vert +\left( d_{k-1}-d_{k}\right) \left\Vert
e_{i}^{0}\right\Vert \right) \\
&\leq &\gamma C_{\max }\left( \dsum\limits_{m=1}^{k-1}\left\vert \sigma
_{m}\right\vert +2-d_{1}+d_{k-1}-d_{k}\right) \left\Vert E^{0}\right\Vert \\
&\leq &\gamma C_{\max }\left( 5-2^{3-\alpha }\right) \left\Vert
E^{0}\right\Vert ,
\end{eqnarray*}
where%
\begin{equation*}
\dsum\limits_{m=1}^{k-1}\left\vert \sigma _{m}\right\vert
+2-d_{1}+d_{k-1}-d_{k}=5-2^{3-\alpha },\text{ }0<\text{ }\sigma _{m}<1,\text{
}-1<d_{k}-d_{k-1}<0,\text{ }1<2-d_{1}<2\text{ }
\end{equation*}
then%
\begin{equation}
\left\Vert E^{k+1}\right\Vert \leq C\left\Vert E^{0}\right\Vert ;\text{ }%
C=C_{\max }\left( 5-2^{3-\alpha }\right) .  \label{58}
\end{equation}
Therefore the method is stable.

\subsection{Convergence}

Let $v(x_{i};t_{k+1})$ as the exact solution and $v_{i}^{k+1}$ is the
approximate solution of scheme $(\ref{38}),$ we put $%
v(x_{i};t_{k+1})-v_{i}^{k+1}=\epsilon _{i}^{k+1};$ for $i=\overline{1,N-1},$ 
$k=\overline{1,M-1}$. The scheme $L_{2}$ defined on $\left( \ref{37}\right) $verified $%
\left( \left[ 26\right] \right) $ 
\begin{equation}
\left\vert \frac{\partial ^{\alpha }v}{\partial t^{\alpha }}-\left( \frac{%
\partial ^{\alpha }v}{\partial t^{\alpha }}\right) _{\text{\ }%
L_{2}}\right\vert \leq O\left( h_{t}\right)  \label{59}
\end{equation}
substitution into $(\ref{38})$ and using $(\ref{35}),$ $(\ref{59})$ leads to%
\begin{equation*}
\gamma \dsum\limits_{j\text{ }=\text{ }0}^{k}\left(
v(x_{i};t_{k-j-1})-\epsilon _{i}^{k-j-1}-2\left( v(x_{i};t_{k-j})-\epsilon
_{i}^{k-j}\right) +\left( v(x_{i};t_{k-j+1})-\epsilon _{i}^{k-j+1}\right)
\right) d_{j}
\end{equation*}%
\begin{equation*}
+L\left( v(x_{i};t_{k+1})-\epsilon _{i}^{k+1}\right) =g_{i}^{\text{ }k+1}.
\end{equation*}
then%
\begin{equation*}
\gamma \dsum\limits_{j\text{ }=\text{ }0}^{k}\left(
v(x_{i};t_{k-j-1})-2v(x_{i};t_{k-j})+\left( v(x_{i};t_{k-j+1})\right)
\right) d_{j}+Lv(x_{i};t_{k+1})
\end{equation*}%
\begin{equation*}
-\gamma \dsum\limits_{j\text{ }=\text{ }0}^{k}\left( \epsilon
_{i}^{k-j-1}-2\epsilon _{i}^{k-j}+\epsilon _{i}^{k-j+1}\right)
d_{j}-L\epsilon _{i}^{k+1}=g_{i}^{\text{ }k+1}
\end{equation*}
so%
\begin{equation*}
\frac{\partial ^{\alpha }v(x,t)}{\partial t^{\alpha }}%
+O(h_{t})+Lv(x;t)+O(h)-\gamma \dsum\limits_{j\text{ }=\text{ }0}^{k}\left(
\epsilon _{i}^{k-j-1}-2\epsilon _{i}^{k-j}+\epsilon _{i}^{k-j+1}\right)
d_{j}-L\epsilon _{i}^{k+1}=g_{i}^{\text{ }k+1}.
\end{equation*}
hence%
\begin{equation}
\gamma \dsum\limits_{j\text{ }=\text{ }0}^{k}\left( \epsilon
_{i}^{k-j-1}-2\epsilon _{i}^{k-j}+\epsilon _{i}^{k-j+1}\right)
d_{j}+L\epsilon _{i}^{k+1}=O(h+h_{t}).  \label{60}
\end{equation}
Taking%
\begin{equation*}
\left\vert \epsilon _{l}^{k}\right\vert =\left\Vert \epsilon ^{k}\right\Vert
=\limfunc{Max}_{\text{ }1\text{ }\leq \text{ }i\text{ }\leq \text{ }%
N-1}\left\vert \epsilon _{i}^{k}\right\vert ;\epsilon ^{k}=\left( \epsilon
_{1}^{k},,...,\epsilon _{N-1}^{k}\right) ^{T};\text{ }\left\Vert \epsilon
_{i}^{0}\right\Vert =0
\end{equation*}
for $k=0$ we get 
\begin{equation}
F_{i}^{1}\epsilon _{i-1}^{1}+A_{i}^{1}\epsilon _{i}^{1}+B_{i}^{1}\epsilon
_{i+1}^{1}=\gamma \epsilon _{i}^{0}+O(h+h_{t})\text{\quad with\quad }i=%
\overline{1,N-1},  \label{61}
\end{equation}
we have%
\begin{eqnarray*}
\left\Vert \epsilon ^{1}\right\Vert &=&\left\vert \epsilon
_{l}^{1}\right\vert \leq \left( F_{i}^{1}+A_{i}^{1}+B_{i}^{1}\right)
\left\vert \epsilon _{l}^{1}\right\vert \\
&\leq &\left( \left( A_{i}^{1}+B_{i}^{1}\right) \left\vert \epsilon
_{l}^{1}\right\vert +F_{i}^{1}\left\vert \epsilon _{l}^{1}\right\vert \right)
\\
&\leq &\limfunc{Max}_{\text{ }1\text{ }\leq \text{ }i\text{ }\leq \text{ }%
N-1}\left\vert F_{i}^{1}\epsilon _{i-1}^{1}+A_{i}^{1}\epsilon
_{l}^{1}+B_{i}^{1}\epsilon _{l}^{1}\right\vert =O(h+h_{t}).
\end{eqnarray*}
hence%
\begin{equation}
\left\Vert \epsilon ^{1}\right\Vert \leq O(h+h_{t}).  \label{62}
\end{equation}
We assume : $\left\vert \epsilon _{l}^{j}\right\vert \leq O(h+h_{t})$; $j=%
\overline{1,k}$ from $(\ref{60})$ we get 
\begin{equation}
F_{i}^{k+1}\epsilon _{i-1}^{k+1}+A_{i}^{k+1}\epsilon
_{i}^{k+1}+B_{i}^{k+1}\epsilon _{i+1}^{k+1}=-\gamma
\dsum\limits_{m=1}^{k-1}\sigma _{m}\epsilon _{i}^{k-m}+\gamma \left(
2-d_{1}\right) \epsilon _{i}^{k}+O(h+h_{t})  \label{63}
\end{equation}
we have%
\begin{eqnarray*}
\left\Vert \epsilon ^{k+1}\right\Vert &\leq &\left( \gamma
+c_{i}^{k+1}\right) \left\vert \epsilon _{l}^{k+1}\right\vert =\left(
F_{i}^{k+1}+A_{i}^{k+1}+B_{i}^{k+1}\right) \left\vert \epsilon
_{l}^{k+1}\right\vert \\
&\leq &\left( F_{i}^{k+1}\left\vert \epsilon _{i-1}^{k+1}\right\vert +\left(
A_{i}^{k+1}+B_{i}^{k+1}\right) \left\vert \epsilon _{l}^{k+1}\right\vert
\right) \\
&\leq &\limfunc{Max}_{\text{ }1\text{ }\leq \text{ }i\text{ }\leq \text{ }%
N-1}\left\vert -\gamma \dsum\limits_{m=1}^{k-1}\sigma _{m}\epsilon
_{i}^{k-m}+\gamma \left( 2-d_{1}\right) \epsilon
_{i}^{k}+O(h+h_{t})\right\vert \\
&\leq &\gamma \dsum\limits_{m=1}^{k-1}\sigma _{m}\left\Vert \epsilon
^{k-m}\right\Vert +\gamma \left( 2-d_{1}\right) \left\Vert \epsilon
^{k}\right\Vert +O(h+h_{t}) \\
&\leq &\gamma \left( \dsum\limits_{m=1}^{k-1}\sigma _{m}+\left(
2-d_{1}\right) \right) O(h+h_{t})+O(h+h_{t})
\end{eqnarray*}
hence
\begin{equation}
\left\Vert \epsilon ^{k+1}\right\Vert \leq \frac{\gamma }{\gamma +c_{i}^{k+1}%
}O(h+h_{t})+\frac{1}{\gamma +c_{i}^{k+1}}O(h+h_{t})\leq O(h+h_{t}).
\label{64}
\end{equation}
Therefore, the method is convergent.

\section{Applications}

In this section, we give some numerical investigation tests.

\begin{example}
We consider a problem $(\ref{1}-\ref{3})$ with $\alpha =\frac{3}{2},\quad a(x,t)=-x-t,\quad $
$\displaystyle b(x,t)=x+t,\quad c(x)=2,\quad g(x,t)=\left(\frac{3}{4}\sqrt{\pi }+2t\sqrt{t}%
\right)e^{x},\quad \phi (x)=\psi (x)=0,$ $\mu (t)=\displaystyle (e-1)t^{\frac{3}{2}},\quad E(t)=t^{\frac{3}{2}}.$
The analytical solution is given by $v(x,t)=t^{\frac{3}{2}}e^{x}.$
\end{example}

The approximate solution $u(x,t)$ with \textbf{A. E} is the absolute error. \
\begin{center}
Table 1. $h=0.1;$ $h_{t}=0.01$  \\

$%
\begin{tabular}[t]{|c|c|}
\hline
$h$ & $\ v^{1}(x,t)$ \\ \hline
$0.1$ & $1.1052e-03$ \\ \hline
$0.2$ & $1.2214e-03$ \\ \hline
$0.3$ & $1.3499e-03$ \\ \hline
$0.4$ & $1.4918e-03$ \\ \hline
$0.5$ & $1.6487e-03$ \\ \hline
$0.6$ & $1.8221e-03$ \\ \hline
$0.7$ & $2.0138e-03$ \\ \hline
$0.8$ & $2.2255e-03$ \\ \hline
$0.9$ & $2.4596e-03$ \\ \hline
\end{tabular}%
\begin{tabular}[t]{|c|}
\hline
\ $\ u^{1}(x,t)$ \\ \hline
$1.3042e-03$ \\ \hline
$1.4523e-03$ \\ \hline
$1.6038e-03$ \\ \hline
$1.7710e-03$ \\ \hline
$1.9558e-03$ \\ \hline
$2.1600e-03$ \\ \hline
$2.3851e-03$ \\ \hline
$2.6235e-03$ \\ \hline
$2.7079e-03$ \\ \hline
\end{tabular}%
\begin{tabular}[t]{|c|}
\hline
\ $\mathbf{A.E}$ \\ \hline
$1.99e-04$ \\ \hline
$2.30e-04$ \\ \hline
$2.53e-04$ \\ \hline
$2.79e-04$ \\ \hline
$3.07e-04$ \\ \hline
$3.38e-04$ \\ \hline
$3.71e-04$ \\ \hline
$3.98e-04$ \\ \hline
$2.48e-04$ \\ \hline
\end{tabular}%
$
\begin{figure}{h}
	\centering
	\includegraphics [width=0.9\textwidth]{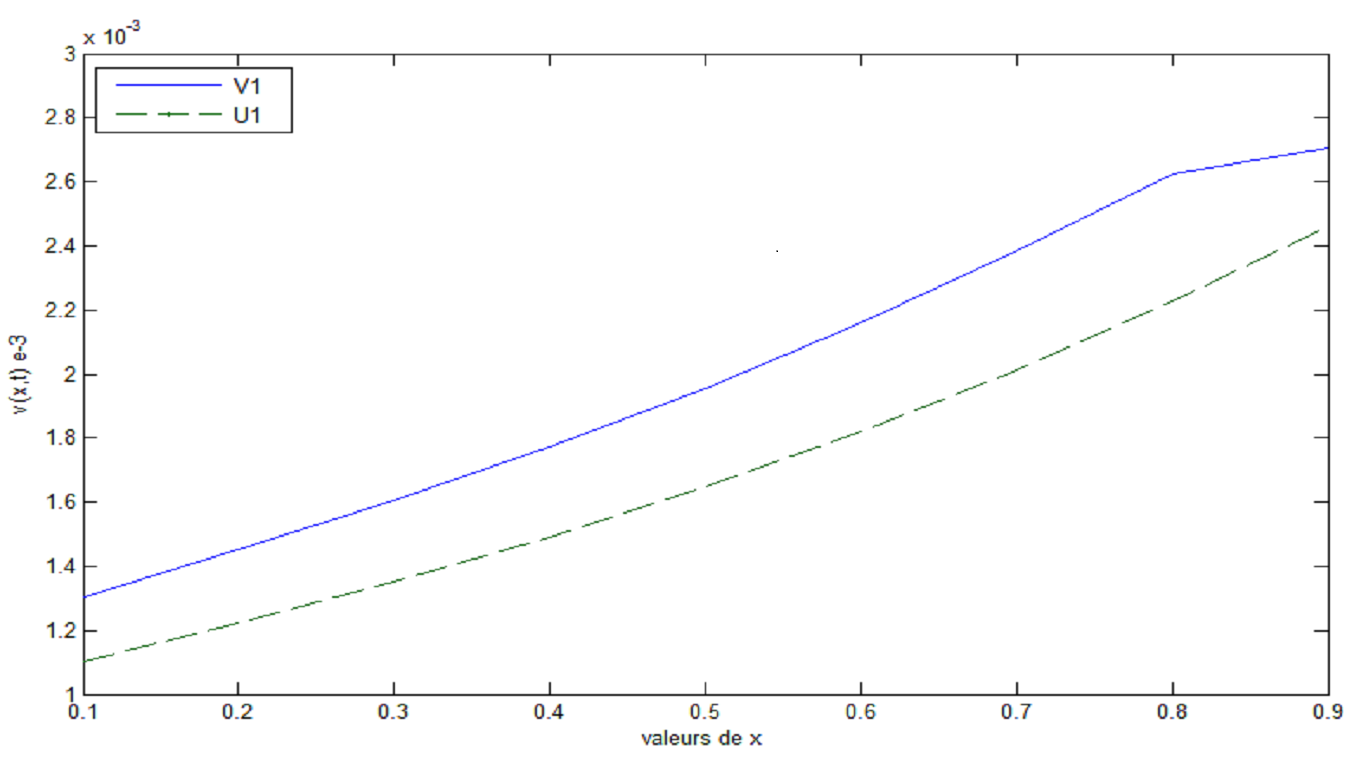}
	\caption{\ $\alpha =1.5$, $h_{t}=0.01$}\label{fig:(1)}
\end{figure}
\end{center}
\newpage 

\begin{center}
	
Table 2. $h=0.1$; $h_{t}\mathbf{=}0.001$ \\

\vspace{0.3in} 
$%
\begin{tabular}{|c|c|}
\hline
$h$ & $\ v^{1}(x,t)$ \\ \hline
$0.1$ & $3.4949e-05$ \\ \hline
$0.2$ & $3.8624e-05$ \\ \hline
$0.3$ & $4.2686e-05$ \\ \hline
$0.4$ & $4.7176e-05$ \\ \hline
$0.5$ & $5.2137e-05$ \\ \hline
$0.6$ & $5.7620e-05$ \\ \hline
$0.7$ & $6.3680e-05$ \\ \hline
$0.8$ & $7.0378e-05$ \\ \hline
$0.9$ & $7.7802e-05$ \\ \hline
\end{tabular}%
\begin{tabular}{|c|}
\hline
$u^{1}(x,t)$ \\ \hline
$4.1175e-05$ \\ \hline
$4.5515e-05$ \\ \hline
$5.0301e-05$ \\ \hline
$5.5590e-05$ \\ \hline
$6.1435e-05$ \\ \hline
$6.7895e-05$ \\ \hline
$7.5034e-05$ \\ \hline
$8.2923e-05$ \\ \hline
$9.1415e-05$ \\ \hline
\end{tabular}%
\begin{tabular}{|l|}
\hline
$\ \ \ \ \ \ \ \mathbf{A.E}$ \\ \hline
\multicolumn{1}{|c|}{$6.23e-06$} \\ \hline
\multicolumn{1}{|c|}{$6.89e-06$} \\ \hline
\multicolumn{1}{|c|}{$7.61e-06$} \\ \hline
\multicolumn{1}{|c|}{$8.41e-06$} \\ \hline
\multicolumn{1}{|c|}{$9.30e-06$} \\ \hline
\multicolumn{1}{|c|}{$1.03e-05$} \\ \hline
\multicolumn{1}{|c|}{$1.14e-05$} \\ \hline
\multicolumn{1}{|c|}{$1.25e-05$} \\ \hline
\multicolumn{1}{|c|}{$1.36e-05$} \\ \hline
\end{tabular}%
$

\unskip
\begin{figure}
	\centering
	\includegraphics [width=0.9\textwidth]{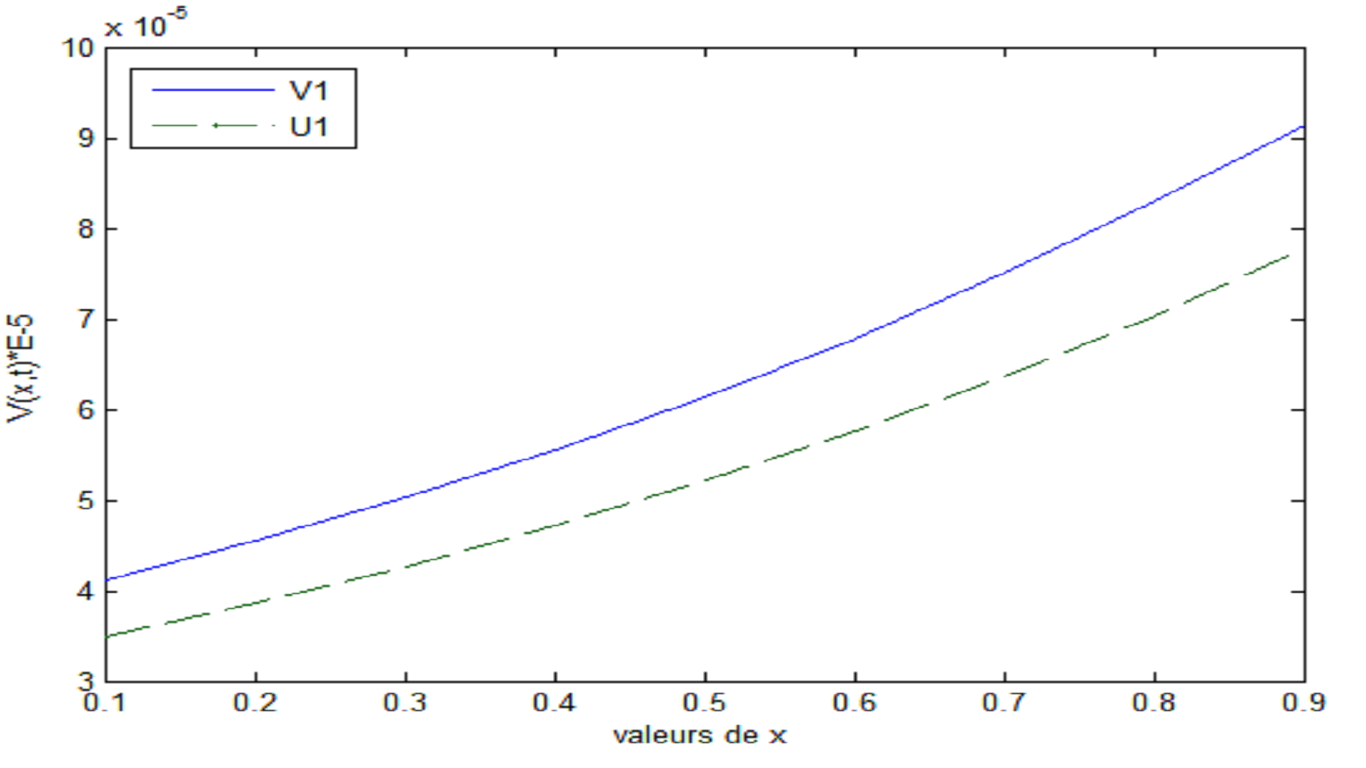}
	\caption{\ $\alpha =1.5$, $h_{t}=0.001$}\label{fig:(2)}
\end{figure}
\end{center} 
\unskip

\begin{center}
Table 3. $h=0.1;$ $h_{t}\mathbf{=0.0001}$ \\

\vspace{0.3in}
$%
\begin{tabular}{|c|c|}
\hline
$h$ & $\ v^{1}(x,t)$ \\ \hline
$0.1$ & $1.1052e-06$ \\ \hline
$0.2$ & $1.2214e-06$ \\ \hline
$0.3$ & $1.3499e-06$ \\ \hline
$0.4$ & $1.4918e-06$ \\ \hline
$0.5$ & $1.6487e-06$ \\ \hline
$0.6$ & $1.8221e-06$ \\ \hline
$0.7$ & $2.0138e-06$ \\ \hline
$0.8$ & $2.2255e-06$ \\ \hline
$0.9$ & $2.4596e-06$ \\ \hline
\end{tabular}%
\begin{tabular}{|c|}
\hline
$u^{1}(x,t)$ \\ \hline
$1.3020e-06$ \\ \hline
$1.4389e-06$ \\ \hline
$1.5903e-06$ \\ \hline
$1.7575e-06$ \\ \hline
$1.9424e-06$ \\ \hline
$2.1466e-06$ \\ \hline
$2.3724e-06$ \\ \hline
$2.6219e-06$ \\ \hline
$2.8974e-06$ \\ \hline
\end{tabular}%
\begin{tabular}{|l|}
\hline
$\ \mathbf{A.E}$ \\ \hline
\multicolumn{1}{|c|}{$1e-07$} \\ \hline
\multicolumn{1}{|c|}{$2e-07$} \\ \hline
\multicolumn{1}{|c|}{$2e-07$} \\ \hline
\multicolumn{1}{|c|}{$2e-07$} \\ \hline
\multicolumn{1}{|c|}{$2e-07$} \\ \hline
\multicolumn{1}{|c|}{$3e-07$} \\ \hline
\multicolumn{1}{|c|}{$3e-07$} \\ \hline
\multicolumn{1}{|c|}{$3e-07$} \\ \hline
\multicolumn{1}{|c|}{$4e-07$} \\ \hline
\end{tabular}%
$

\begin{figure}
	\centering
	\includegraphics [width=0.8\textwidth]{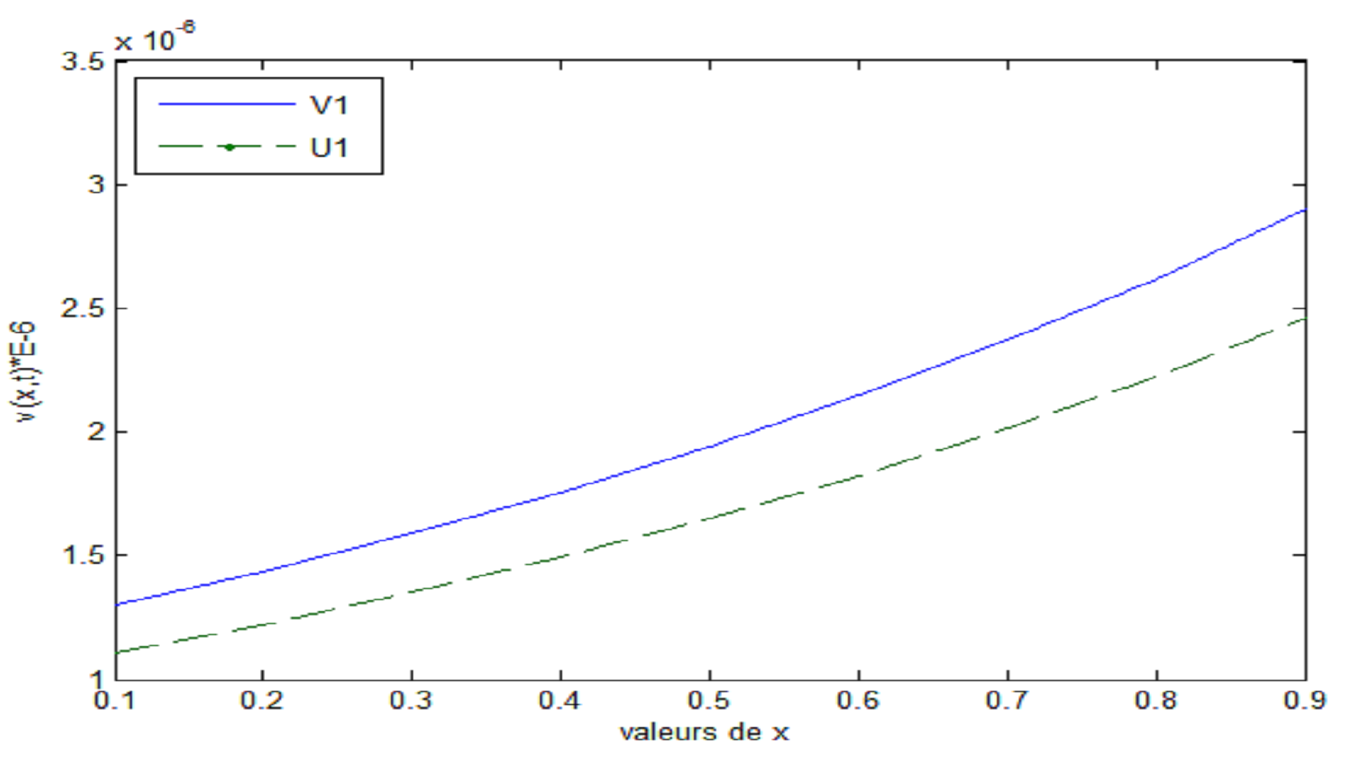}
	\caption{ \ $\alpha =1.5$, $h_{t}=0.0001$}\label{fig:(3)}
\end{figure}
%\unskip
\end{center}
We see in Figures 1, 2 and 3 that the absolute errors(A.E.) decreases when the
step $h_{t}$ takes small values very close to zero. that is, for $%
h_{t}=0.01$, $h_{t}=0.001$, $h_{t}=0.0001$ A.E decreases towards zero and
the approximate solution tends towards the exact solution with convergence
order of $O(h+h_{t}).$
\newpage

For $k=1$ \ $\left( \text{second iteration}\right) $
%\newpage
\begin{center}
Table 4 shows the absolute error for space step $h=0.1$.\\ 

\vspace{0.3in}
$%
\begin{tabular}{|l|l|}
\hline
$h$ & $\ h_{t}=\mathbf{10}^{-2}\ \ \ $ \\ \hline
$0.1$ & \multicolumn{1}{|c|}{$1.84e-03$} \\ \hline
$0.2$ & \multicolumn{1}{|c|}{$1.74e-03$} \\ \hline
$0.3$ & \multicolumn{1}{|c|}{$1.62e-03$} \\ \hline
$0.4$ & \multicolumn{1}{|c|}{$1.49e-03$} \\ \hline
$0.5$ & \multicolumn{1}{|c|}{$1.33e-03$} \\ \hline
$0.6$ & \multicolumn{1}{|c|}{$1.17e-03$} \\ \hline
$0.7$ & \multicolumn{1}{|c|}{$9.82e-04$} \\ \hline
$0.8$ & \multicolumn{1}{|c|}{$7.40e-04$} \\ \hline
$0.9$ & \multicolumn{1}{|c|}{$1.26e-04$} \\ \hline
\end{tabular}%
\begin{tabular}{|l|}
\hline
$\ \ h_{t}=\mathbf{10}^{-3}$ \  \\ \hline
\multicolumn{1}{|c|}{$5.80e-05$} \\ \hline
\multicolumn{1}{|c|}{$5.45e-05$} \\ \hline
\multicolumn{1}{|c|}{$5.06e-05$} \\ \hline
\multicolumn{1}{|c|}{$4.63e-05$} \\ \hline
\multicolumn{1}{|c|}{$4.15e-05$} \\ \hline
\multicolumn{1}{|c|}{$3.63e-05$} \\ \hline
\multicolumn{1}{|c|}{$3.05e-05$} \\ \hline
\multicolumn{1}{|c|}{$2.40e-05$} \\ \hline
\multicolumn{1}{|c|}{$1.64e-05$} \\ \hline
\end{tabular}%
\begin{tabular}{|l|}
\hline
\ $\ \ h_{t}=\mathbf{10}^{-5}$ \  \\ \hline
\multicolumn{1}{|c|}{$5.80e-08$} \\ \hline
\multicolumn{1}{|c|}{$5.45e-08$} \\ \hline
\multicolumn{1}{|c|}{$5.06e-08$} \\ \hline
\multicolumn{1}{|c|}{$4.63e-08$} \\ \hline
\multicolumn{1}{|c|}{$4.16e-08$} \\ \hline
\multicolumn{1}{|c|}{$3.63e-08$} \\ \hline
\multicolumn{1}{|c|}{$3.05e-08$} \\ \hline
\multicolumn{1}{|c|}{$2.41e-08$} \\ \hline
\multicolumn{1}{|c|}{$1.69e-08$} \\ \hline
\end{tabular}%
$%

\begin{figure}
	\begin{subfigure}{.3\textwidth}
		\centering
		\includegraphics[width=3.5cm,height=4.5cm]{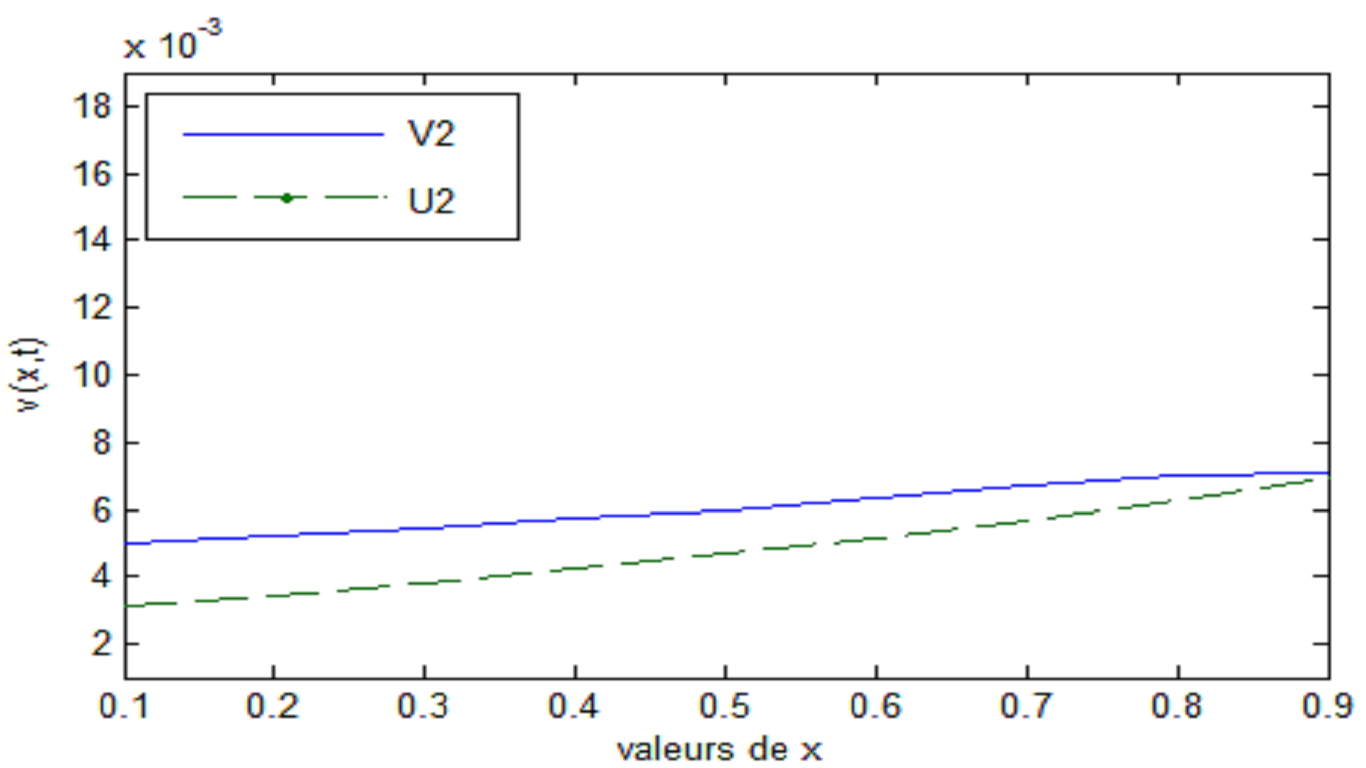}
		\caption{FIGURE 4.\ \ \ \ }$h_{t}=\mathbf{10}^{-2}$\label{fig:(2A)}
	\end{subfigure}\hfill
	\begin{subfigure}{.3\textwidth}
		\centering
		\includegraphics[width=3.5cm,height=4cm]{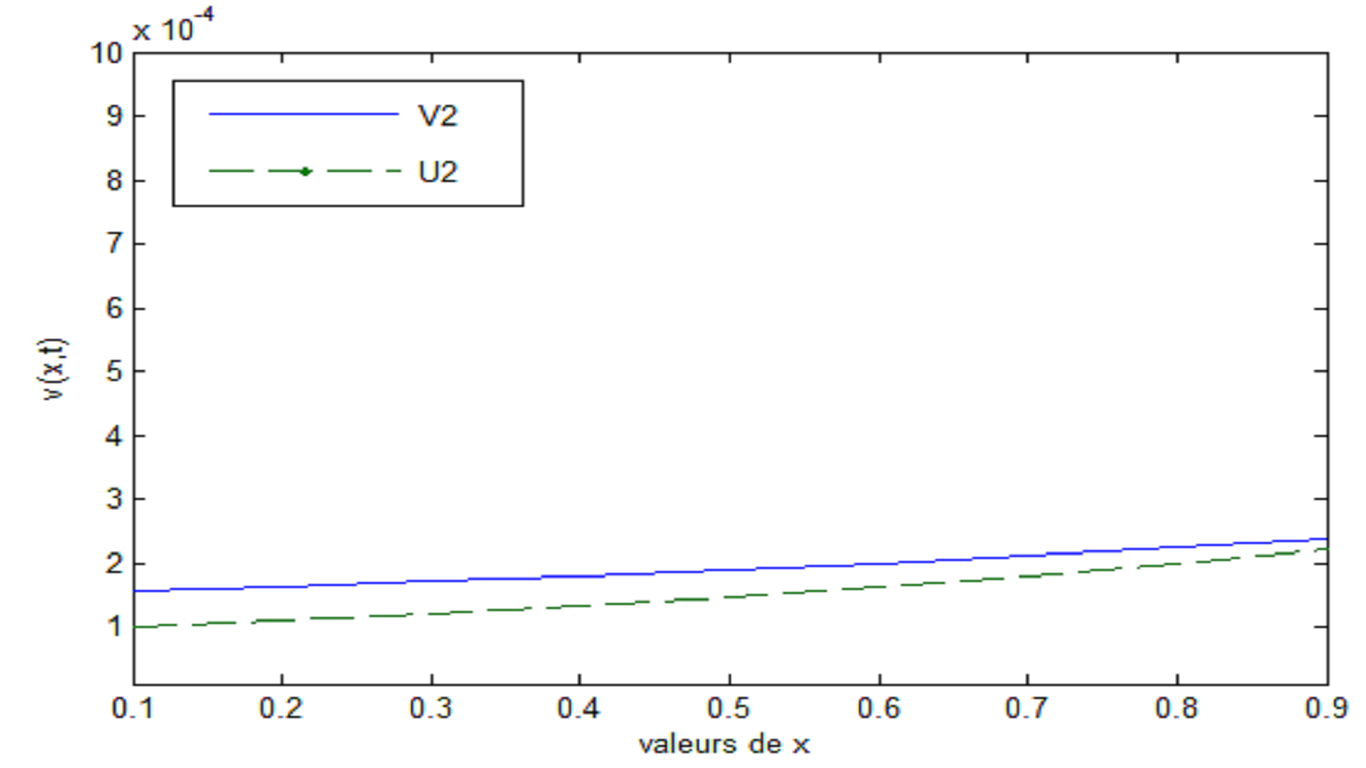}
		\caption{FIGURE 5.\ \ \ \ \ }$h_{t}=\mathbf{10}^{-3}$\label{fig:(2B)}
	\end{subfigure}\hfill
	\begin{subfigure}{.3\textwidth}
		\centering
		\includegraphics[width=3.5cm,height=4cm]{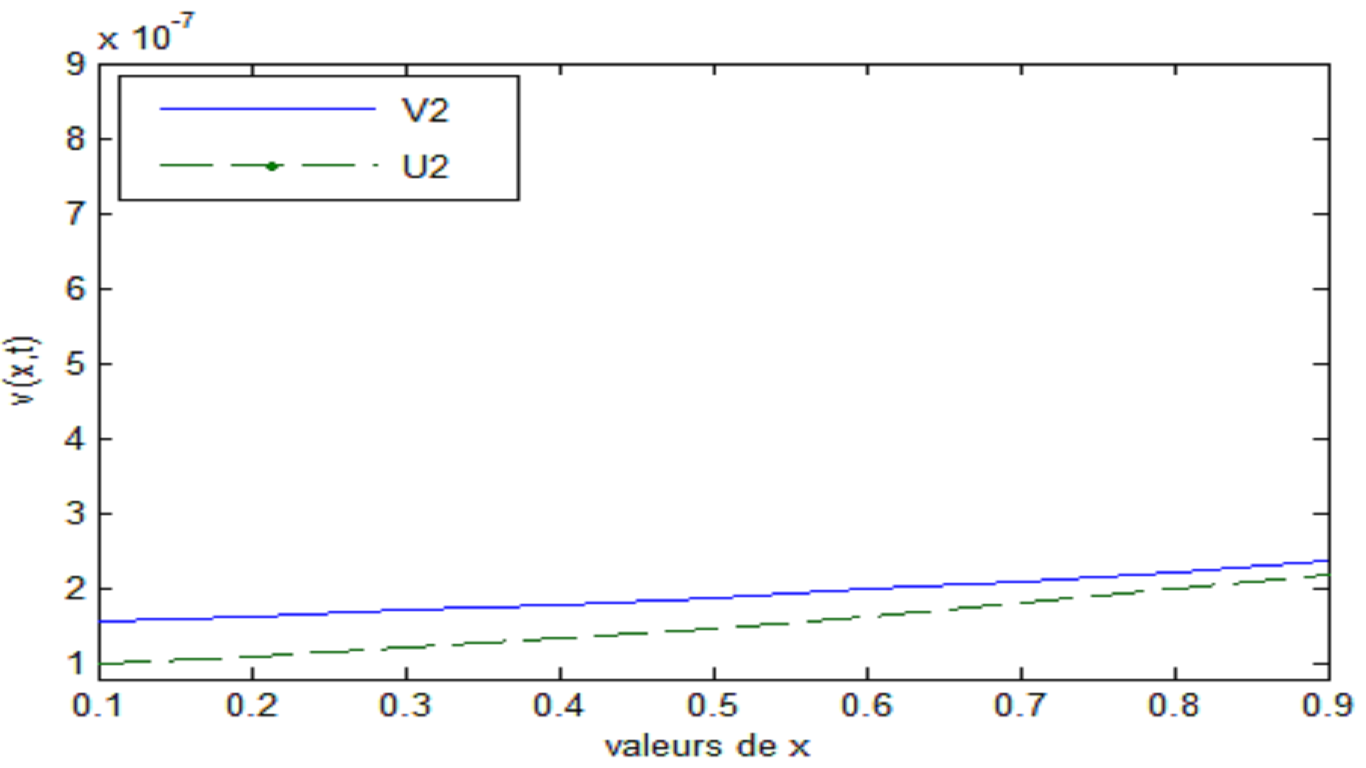}
		\caption{FIGURE 6.\ \ \ \ \ \ }$h_{t}=\mathbf{10}^{-5}$\label{fig:(2C)}
	\end{subfigure}
\end{figure}
\end{center}

Table 4 shows the absolute error decreases to zero and Fig 4,5, and 6 show
the approximate solution $u{{}^2}$ after two steps $2h_{t}$ tends to the exact solution when $h_{t}$
close to zero, with convergence order $O(h+h_{t}).$

\begin{center}
$\overset{%
\begin{array}{cc}
& \text{Table 5. The absolute error for }h\mathbf{=0.01};\text{ }h_{t}=%
\mathbf{10}^{-3}\ \text{\ } \\ 
& 
\begin{tabular}[t]{|c|c|c|c|c|c|c|c|c|c|c|}
\hline
$i=\overline{1,9}$ & $\overline{10,18}$ & $\overline{19,27}$ & $\overline{%
28,36}$ & $\overline{37,45}$ & $\overline{46,54}$ & $\overline{55,63}$ & $%
\overline{64,72}$ & $\overline{73,81}$ & $\overline{82,89}$ & $\overline{%
90,99}$ \\ \hline
5*10$^{-6}$ & 6*10$^{-6}$ & 6*10$^{-6}$ & 7*10$^{-6}$ & 8*10$^{-6}$ & 9*10$%
^{-6}$ & 10$^{-5}$ & 10$^{-5}$ & 10$^{-5}$ & 10$^{-5}$ & 10$^{-5}$ \\ \hline
5*10$^{-6}$ & 6*10$^{-6}$ & 7*10$^{-6}$ & 7*10$^{-6}$ & 8*10$^{-6}$ & 9*10$%
^{-6}$ & 10$^{-5}$ & 10$^{-5}$ & 10$^{-5}$ & 10$^{-5}$ & 10$^{-5}$ \\ \hline
5*10$^{-6}$ & 6*10$^{-6}$ & 7*10$^{-6}$ & 7*10$^{-6}$ & 8*10$^{-6}$ & 9*10$%
^{-6}$ & 10$^{-5}$ & 10$^{-5}$ & 10$^{-5}$ & 10$^{-5}$ & 10$^{-5}$ \\ \hline
6*10$^{-6}$ & 6*10$^{-6}$ & 7*10$^{-6}$ & 7*10$^{-6}$ & 8*10$^{-6}$ & 9*10$%
^{-6}$ & 10$^{-5}$ & 10$^{-5}$ & 10$^{-5}$ & 10$^{-5}$ & 10$^{-5}$ \\ \hline
6*10$^{-6}$ & 6*10$^{-6}$ & 7*10$^{-6}$ & 7*10$^{-6}$ & 8*10$^{-6}$ & 9*10$%
^{-6}$ & 10$^{-5}$ & 10$^{-5}$ & 10$^{-5}$ & 10$^{-5}$ & 10$^{-5}$ \\ \hline
6*10$^{-6}$ & 6*10$^{-6}$ & 7*10$^{-6}$ & 7*10$^{-6}$ & 8*10$^{-6}$ & 9*10$%
^{-6}$ & 10$^{-5}$ & 10$^{-5}$ & 10$^{-5}$ & 10$^{-5}$ & 10$^{-5}$ \\ \hline
6*10$^{-6}$ & 6*10$^{-6}$ & 7*10$^{-6}$ & 8*10$^{-6}$ & 8*10$^{-6}$ & 9*10$%
^{-6}$ & 10$^{-5}$ & 10$^{-5}$ & 10$^{-5}$ & 10$^{-5}$ & 10$^{-5}$ \\ \hline
6*10$^{-6}$ & 6*10$^{-6}$ & 7*10$^{-6}$ & 8*10$^{-6}$ & 8*10$^{-6}$ & 9*10$%
^{-6}$ & 10$^{-5}$ & 10$^{-5}$ & 10$^{-5}$ & 10$^{-5}$ & 10$^{-5}$ \\ \hline
6*10$^{-6}$ & 6*10$^{-6}$ & 7*10$^{-6}$ & 8*10$^{-6}$ & 8*10$^{-6}$ & 9*10$%
^{-6}$ & 10$^{-5}$ & 10$^{-5}$ & 10$^{-5}$ & 10$^{-5}$ & 10$^{-5}$ \\ \hline
\end{tabular}%
\end{array}%
}{}$ 

\begin{figure}
	\centering
	\includegraphics [width=0.9\textwidth]{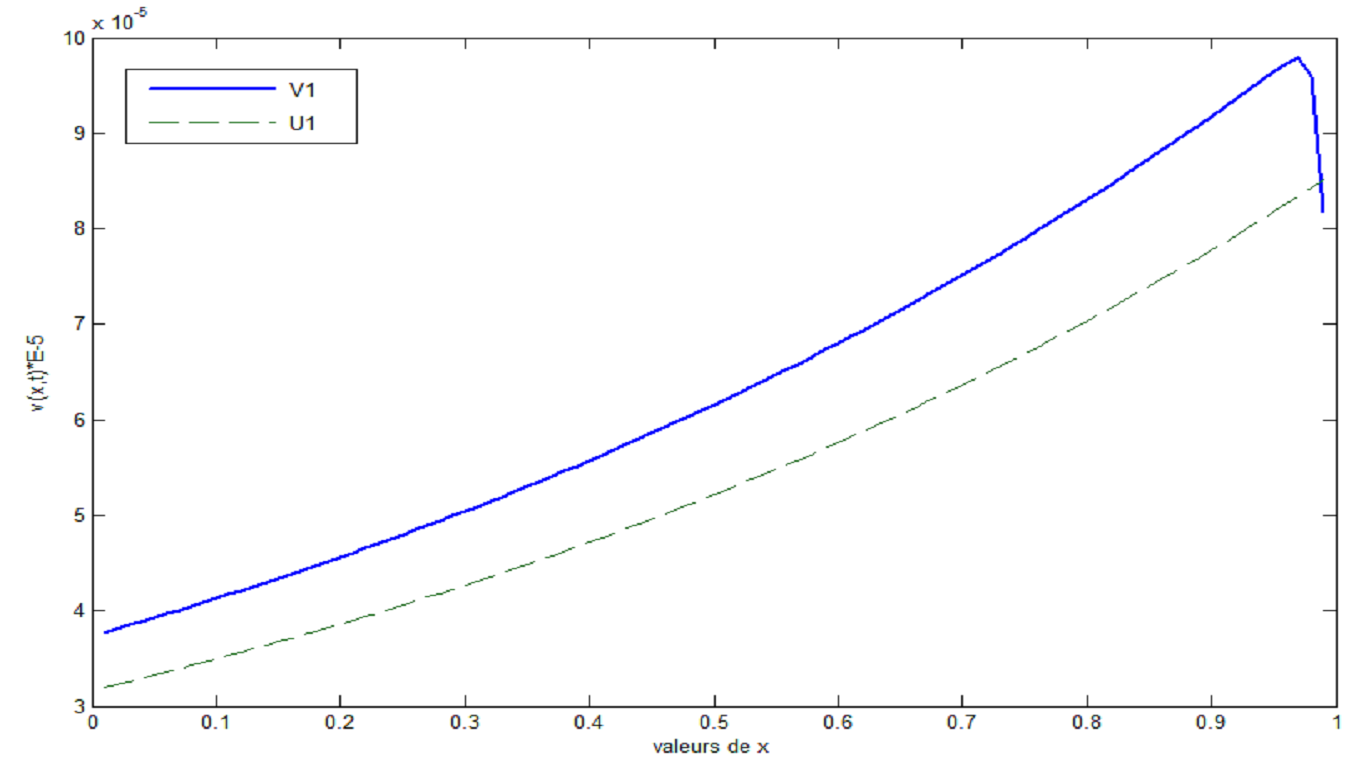}
	{FIGURE 7,  $\alpha =1.5$, $h_{t}=10^{-3}$}\label{fig:(7)}
\end{figure}
\end{center}
\begin{center}
	
$\bigskip \overset{%
\begin{array}{ll}
& \text{Table 6. The absolute error for }h=0.01;\text{ }h_{t}\mathbf{=10}%
^{-5}\ \text{\ } \\ \vspace{0.5in}
& 
\begin{tabular}[t]{|c|c|c|c|c|c|c|c|c|c|c|}
\hline
$i=\overline{1,9}$ & $\overline{10,18}$ & $\overline{19,27}$ & $\overline{%
28,36}$ & $\overline{37,45}$ & $\overline{46,54}$ & $\overline{55,63}$ & $%
\overline{64,72}$ & $\overline{73,81}$ & $\overline{82,89}$ & $\overline{%
90,99}$ \\ \hline
5*10$^{-9}$ & 6*10$^{-9}$ & 6*10$^{-9}$ & 7*10$^{-9}$ & 8*10$^{-9}$ & 8*10$%
^{-9}$ & 9*10$^{-9}$ & 10$^{-8}$ & 10$^{-8}$ & 10$^{-8}$ & 10$^{-8}$ \\ 
\hline
5*10$^{-9}$ & 6*10$^{-9}$ & 6*10$^{-9}$ & 7*10$^{-9}$ & 8*10$^{-9}$ & 9*10$%
^{-9}$ & 9*10$^{-9}$ & 10$^{-8}$ & 10$^{-8}$ & 10$^{-8}$ & 10$^{-8}$ \\ 
\hline
5*10$^{-9}$ & 6*10$^{-9}$ & 6*10$^{-9}$ & 7*10$^{-9}$ & 8*10$^{-9}$ & 9*10$%
^{-9}$ & 9*10$^{-9}$ & 10$^{-8}$ & 10$^{-8}$ & 10$^{-8}$ & 10$^{-8}$ \\ 
\hline
5*10$^{-9}$ & 6*10$^{-9}$ & 7*10$^{-9}$ & 7*10$^{-9}$ & 8*10$^{-9}$ & 9*10$%
^{-9}$ & 10$^{-9}$ & 10$^{-8}$ & 10$^{-8}$ & 10$^{-8}$ & 10$^{-8}$ \\ \hline
5*10$^{-9}$ & 6*10$^{-9}$ & 7*10$^{-9}$ & 7*10$^{-9}$ & 8*10$^{-9}$ & 9*10$%
^{-9}$ & 10$^{-9}$ & 10$^{-8}$ & 10$^{-8}$ & 10$^{-8}$ & 10$^{-8}$ \\ \hline
5*10$^{-9}$ & 6*10$^{-9}$ & 7*10$^{-9}$ & 7*10$^{-9}$ & 8*10$^{-9}$ & 9*10$%
^{-9}$ & 10$^{-9}$ & 10$^{-8}$ & 10$^{-8}$ & 10$^{-8}$ & 10$^{-8}$ \\ \hline
6*10$^{-9}$ & 6*10$^{-9}$ & 7*10$^{-9}$ & 7*10$^{-9}$ & 8*10$^{-9}$ & 9*10$%
^{-9}$ & 10$^{-9}$ & 10$^{-8}$ & 10$^{-8}$ & 10$^{-8}$ & 10$^{-8}$ \\ \hline
6*10$^{-9}$ & 6*10$^{-9}$ & 7*10$^{-9}$ & 7*10$^{-9}$ & 8*10$^{-9}$ & 9*10$%
^{-9}$ & 10$^{-9}$ & 10$^{-8}$ & 10$^{-8}$ & 10$^{-8}$ & 10$^{-8}$ \\ \hline
6*10$^{-9}$ & 6*10$^{-9}$ & 7*10$^{-9}$ & 8*10$^{-9}$ & 8*10$^{-9}$ & 9*10$%
^{-9}$ & 10$^{-9}$ & 10$^{-8}$ & 10$^{-8}$ & 10$^{-8}$ & 10$^{-8}$ \\ \hline
\end{tabular}%
\end{array}%
}{}$%
\end{center}

\begin{figure}
	\centering
	\includegraphics [width=0.9\textwidth]{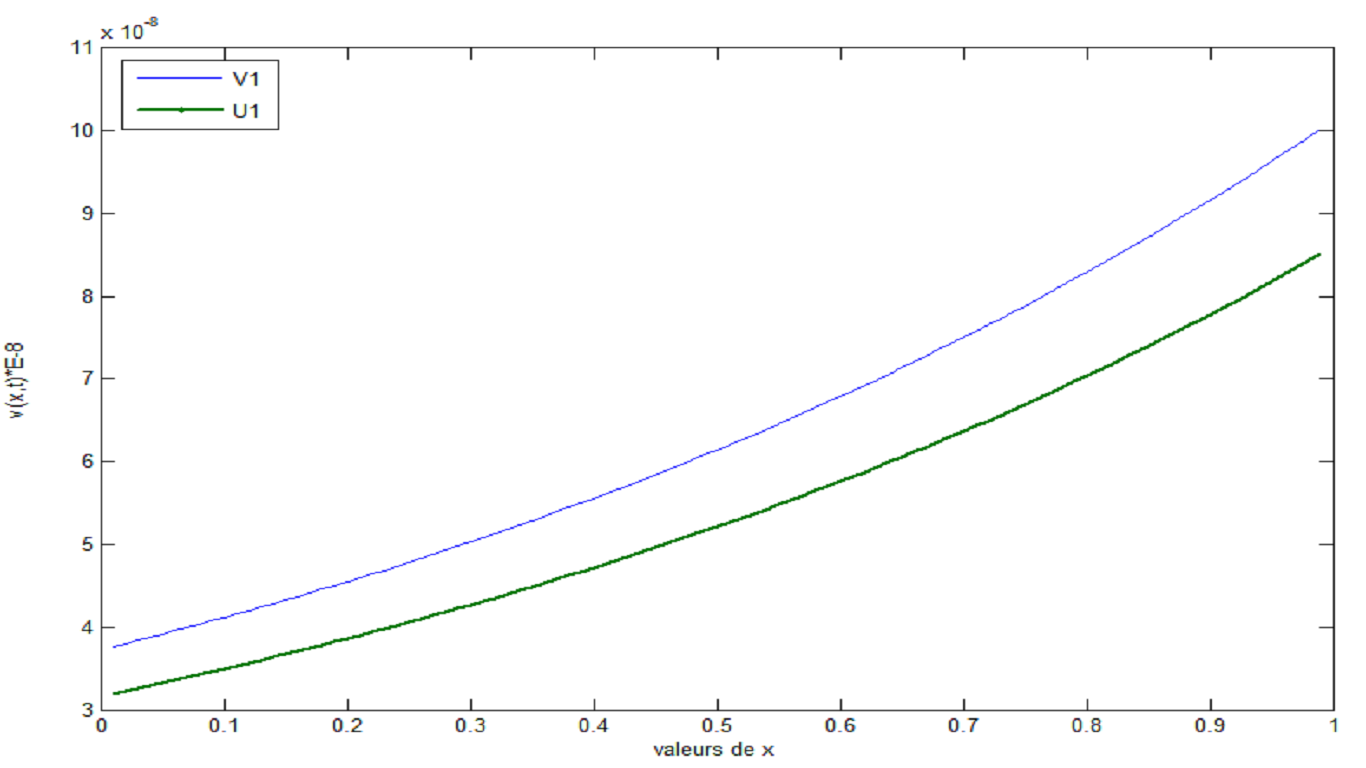} \\
	{FIGURE 8 \ $\alpha =1.5$, $h_{t}=10^{-5}$}\label{fig:(8)}
\end{figure}

From tables 5, 6 and Fig 7, 8 with space step $h=\mathbf{0.01},$ we see that
the approximate solution $u{{}^1}$ tends to the exact solution $v{{}^1}$
when the step $h_{t}$ $(h_{t}=10^{-3},h_{t}=10^{-5})$ takes values close to zero,with convergence order $O(h+h_{t}).$

\begin{example}
We take: $\alpha =\frac{3}{2},\quad a(x,t)=-x^{2}-t,\quad b(x,t)=x-t$,\quad \ 
$c(x)=x+2t,\quad \displaystyle g(x,t)=(4\sqrt{t}+(t+1)^{2}(x^{2}+2t)e^{x},$ $\Phi
(x)=e^{x};\quad $, $\psi (x)=2e^{x},\quad \mu(t)=(t+1)^{2},\quad E(t)=(t+1)^{2}.$\\

The exact analytical solution of this problem is given by $%
v(x,t)=(t+1)^{2}e^{x}.$
\end{example}
%\begin{center}
\newpage
The tables 7, 8 and 9 show the values of the absolute error.
\begin{center}
$%
\begin{array}{c}
\text{Table 7.\ }h\text{ }\mathbf{=0.1},\text{ }h_{t}=\mathbf{10}^{-2} \\ 
\begin{tabular}{|l|l|}
\hline
$h$ & \ \ \ \ $\mathbf{A.E}$ \\ \hline
$0.1$ & \multicolumn{1}{|c|}{$6.72e-04$} \\ \hline
$0.2$ & \multicolumn{1}{|c|}{$2.47e-03$} \\ \hline
$0.3$ & \multicolumn{1}{|c|}{$2.44e-03$} \\ \hline
$0.4$ & \multicolumn{1}{|c|}{$2.41e-03$} \\ \hline
$0.5$ & \multicolumn{1}{|c|}{$2.39e-03$} \\ \hline
$0.6$ & \multicolumn{1}{|c|}{$2.38e-03$} \\ \hline
$0.7$ & \multicolumn{1}{|c|}{$1.84e-03$} \\ \hline
$0.8$ & \multicolumn{1}{|c|}{$1.40e-02$} \\ \hline
$0.9$ & \multicolumn{1}{|c|}{$2.06e-01$} \\ \hline
\end{tabular}%
\end{array}%
$ \end{center}

\begin{center}
\begin{figure}
\centering
\includegraphics [width=0.9\textwidth]{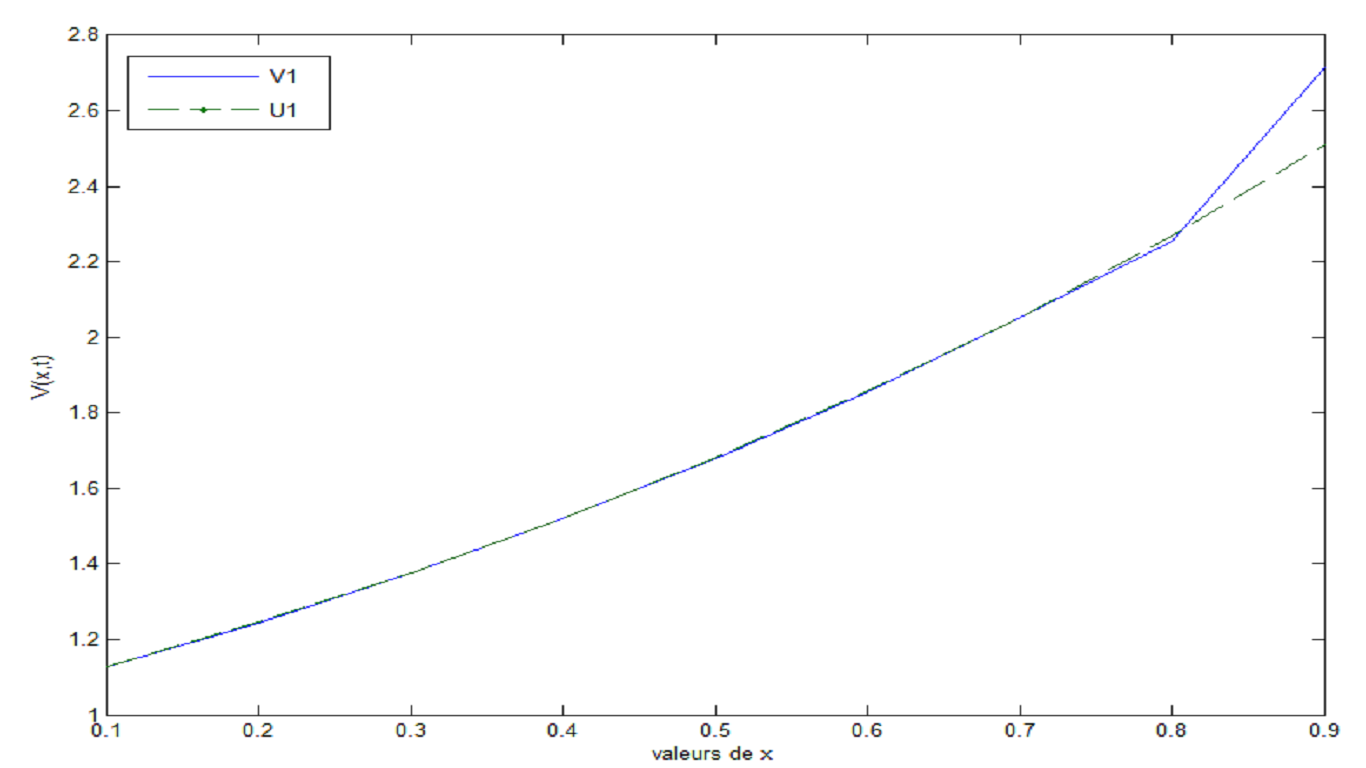} \\
{FIGURE 9. \ $\alpha =1.5$, $h_{t}=10^{-4}$}\label{fig:(9)}
\end{figure}
\end{center}
$%
\begin{array}{c}
\text{Table 8. }h=\mathbf{0.1},\text{ }h_{t}=\mathbf{10}^{-3} \\ 
\begin{tabular}{|l|l|}
\hline
$h$ & $\ \ \ \ \ \ \mathbf{A.E}$ \\ \hline
$0.1$ & \multicolumn{1}{|c|}{$1.22e-05$} \\ \hline
$0.2$ & \multicolumn{1}{|c|}{$4.28e-05$} \\ \hline
$0.3$ & \multicolumn{1}{|c|}{$4.24e-05$} \\ \hline
$0.4$ & \multicolumn{1}{|c|}{$4.20e-05$} \\ \hline
$0.5$ & \multicolumn{1}{|c|}{$4.15e-05$} \\ \hline
$0.6$ & \multicolumn{1}{|c|}{$4.10e-05$} \\ \hline
$0.7$ & \multicolumn{1}{|c|}{$4.03e-05$} \\ \hline
$0.8$ & \multicolumn{1}{|c|}{$4.83e-05$} \\ \hline
$0.9$ & \multicolumn{1}{|c|}{$5.48e-03$} \\ \hline
\end{tabular}%
\end{array}%
$
\begin{figure}
	\centering
	\includegraphics [width=0.9\textwidth]{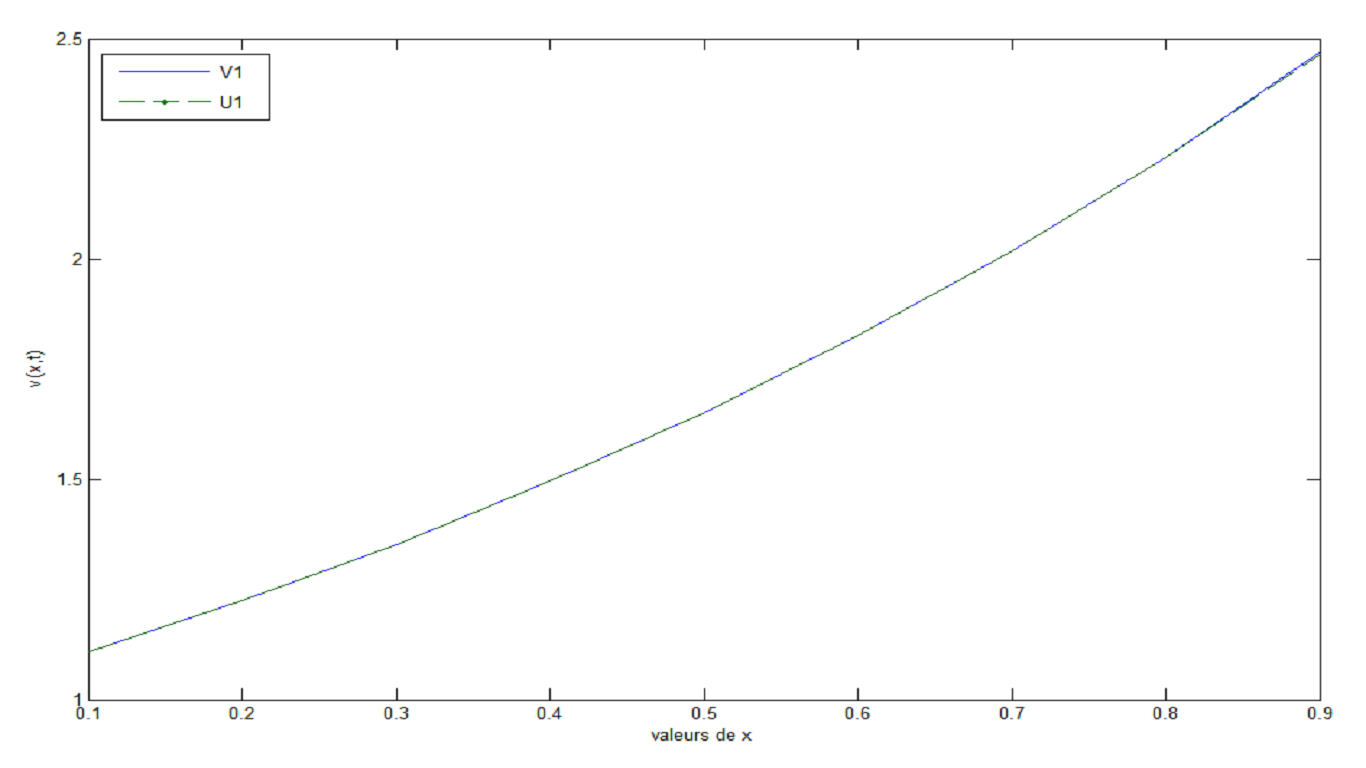} \\
	{FIGURE 10 \ $\alpha =1.5$, $h_{t}=10^{-5}$}\label{fig:(8)}
\end{figure}
%\end{center}

\begin{center}
$
\begin{array}{c}
\text{Table 9. }h=0.1,\text{ }h_{t}=\mathbf{10}^{-4} \\ 
\multicolumn{1}{l}{\text{ \ \ \ \ \ \ }%
\begin{tabular}[t]{|r|r|}
\hline
$h$ & $\mathbf{A.E}$ \ \ \ \ \ \  \\ \hline
$0.1$ & $3.75e-07$ \\ \hline
$0.2$ & $1.26e-06$ \\ \hline
$0.3$ & $1.24e-06$ \\ \hline
$0.4$ & $1.25e-06$ \\ \hline
$0.5$ & $1.23e-06$ \\ \hline
$0.6$ & $1.22e-06$ \\ \hline
$0.7$ & $1.20e-06$ \\ \hline
$0.8$ & $1.19e-06$ \\ \hline
$0.9$ & $1.72e-04$ \\ \hline
\end{tabular}%
}%
\end{array}%
$
\end{center}

\begin{figure}
	\centering
	\includegraphics [width=0.9\textwidth]{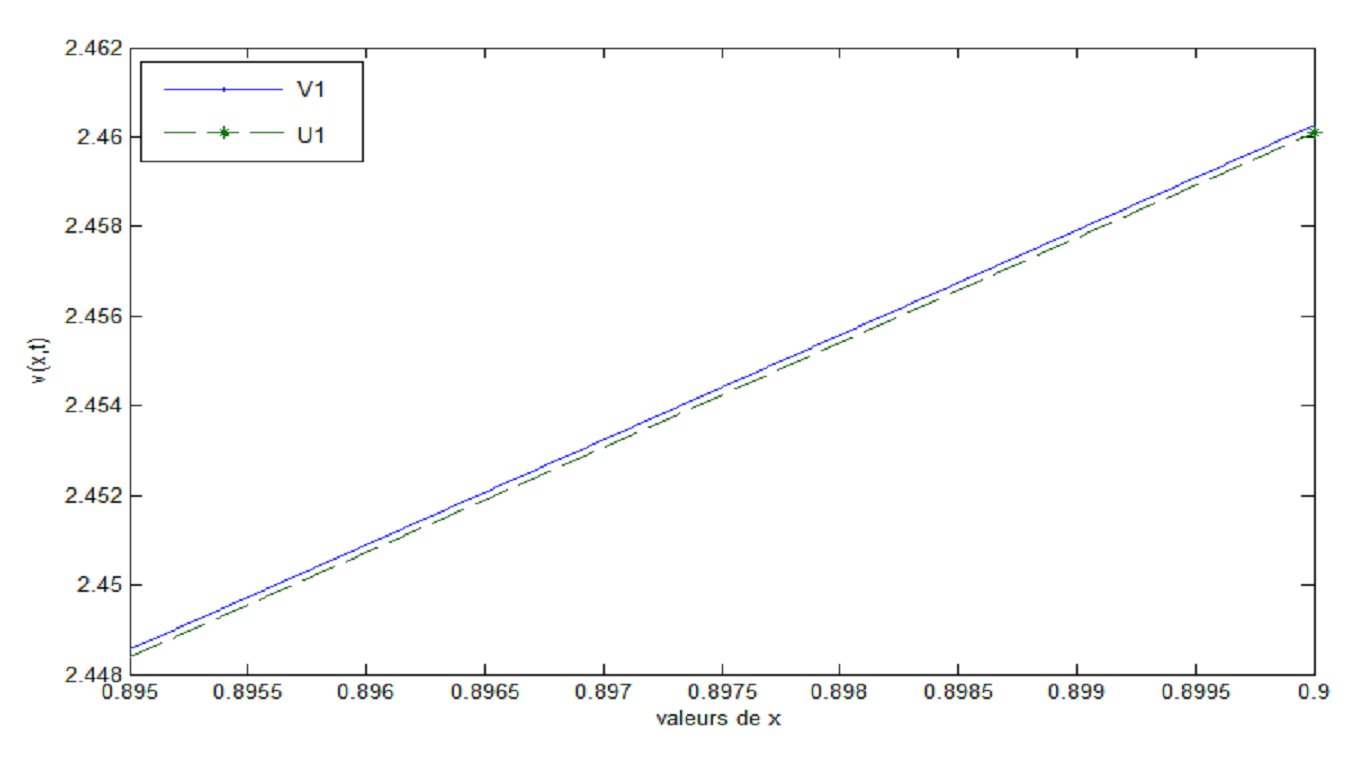} \\
	{FIGURE 11 \ $\alpha =1.5$, $h_{t}=10^{-5}$}\label{fig:(8)}
\end{figure}

In this example we see again for space step $h=0.1$ the absolute error tends
to zero, when the time step $h_{t}$ $(\mathbf{10}^{-2},$ $\mathbf{10}^{-3},$ 
$\mathbf{10}^{-4})$ takes a value close to zero, with convergence order $%
O(h+h_{t}).$

For\textbf{\ }$h=\mathbf{0.01,}$ $\alpha =1.5$%

\begin{figure}
	\begin{subfigure}{.33\textwidth}
		\centering
		\includegraphics[width=3.5cm,height=4.5cm]{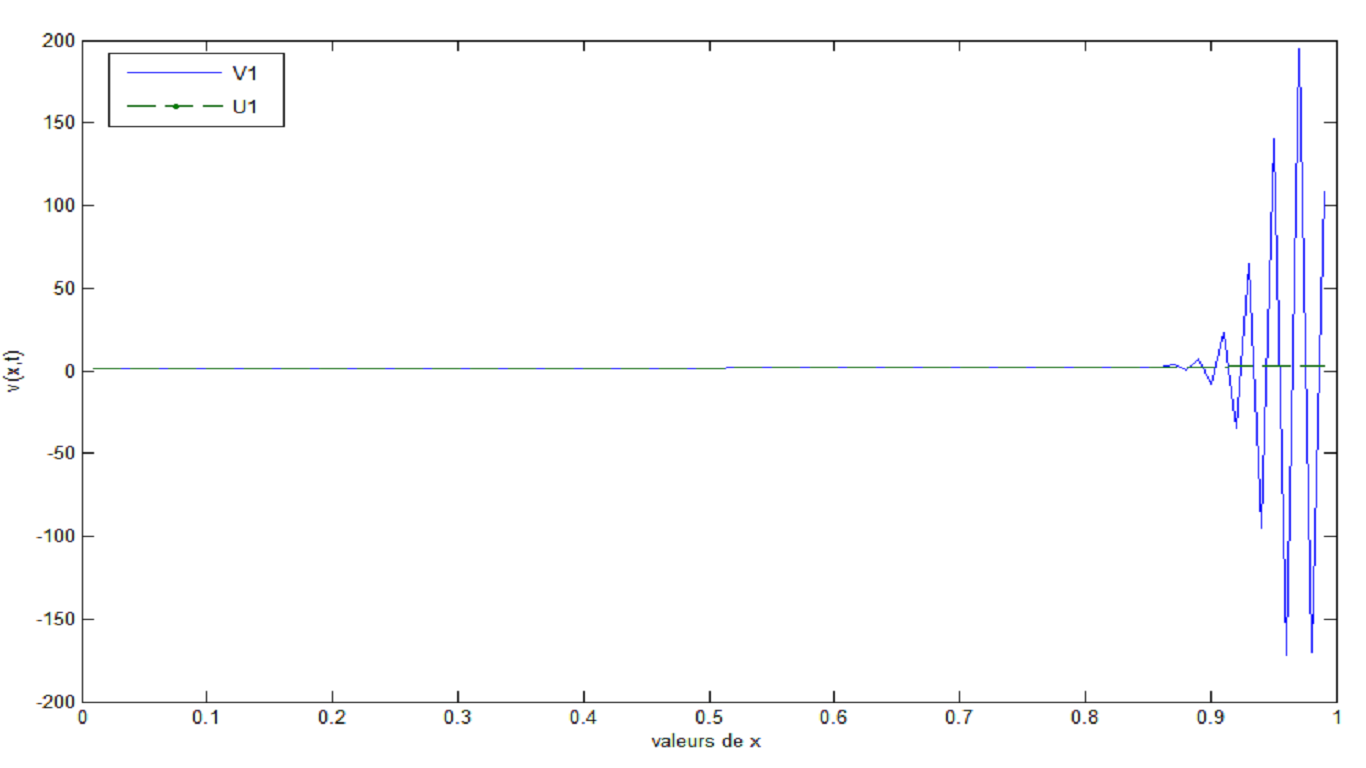}
		\caption{FIGURE 12.\ \ \ \ }$h_{t}=\mathbf{10}^{-2}$\label{fig:(12A)}
	\end{subfigure}\hfill
	\begin{subfigure}{.33\textwidth}
		\centering
		\includegraphics[width=3.5cm,height=4cm]{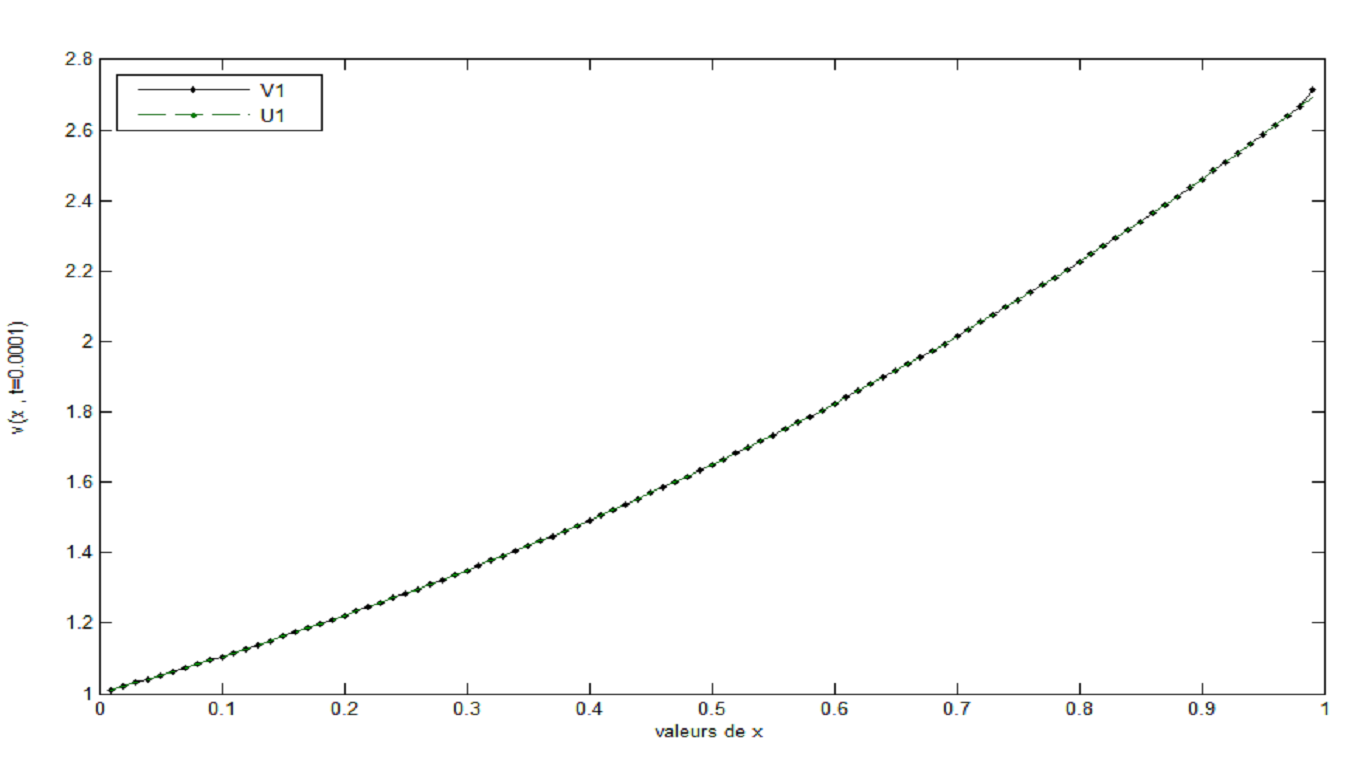}
		\caption{FIGURE 13.\ \ \ \ \ }$h_{t}=\mathbf{10}^{-3}$\label{fig:(12B)}
	\end{subfigure}\hfill
	\begin{subfigure}{.33\textwidth}
		\centering
		\includegraphics[width=3.5cm,height=4cm]{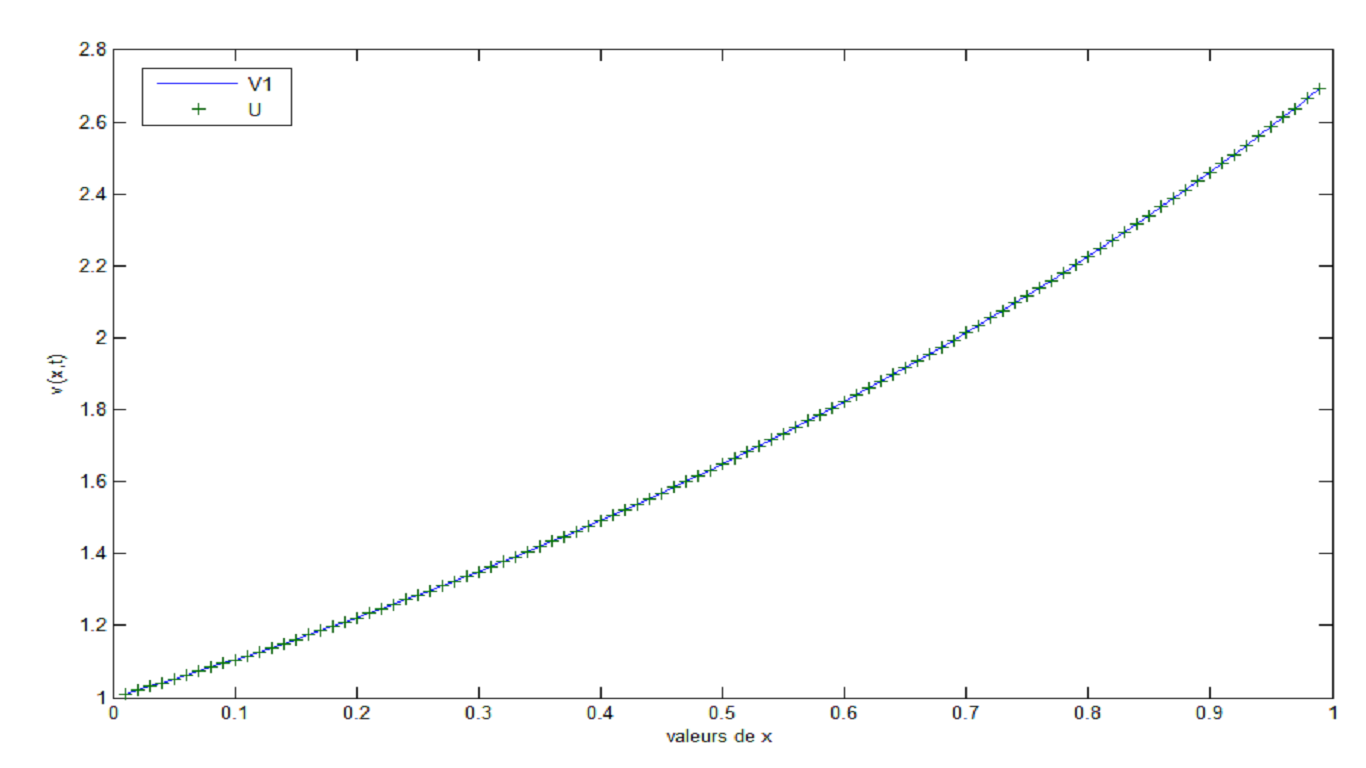}
		\caption{FIGURE 14.\ \ \ \ \ \ }$h_{t}=\mathbf{10}^{-5}$\label{fig:(12C)}
	\end{subfigure}
	{\ Fig. 12\ \ \ \ \ \ $h_{t}=0.001$ \ \ \ \ \ \ \ \ \ \ \ \ \ Fig.
		13\ \ \ \ \  $h_{t}=0.0001$\ \ \ \ \ \ \ \ \ \ \ \ \ \ \ \ Fig. 14\ \ \
		\ }$h_{t}=0.00001$
\end{figure}

The Fig.12 ,13 and 14 show where the space step is fixed at $h=0.01$ and the
time step $h_{t}$ decreases towards zero $(h_{t}=0.001$, $h_{t}=0.0001,$ $%
h_{t}=0.00001),$ the approximate solution $u{{}^1}$ tends to the exact solution $v{{}^1}$, in the case where $h_{t}=0.00001$ we see that the two curves of $u{{}^1}$ and $v{{}^1}$ are almost identical.

\noindent  Table $10$ shows the error norm $\left\Vert E^{k}\right\Vert
_{\infty }$ for defferent value of $\alpha $ defined by%
\begin{equation*}
\left\Vert E^{k}\right\Vert _{\infty }=\limfunc{Max}_{\text{ }1\text{ }\leq 
\text{ }i\text{ }\leq \text{ }N-1}\dsum\limits_{i=1}^{N-1}\left\vert
e_{i}\right\vert ,\text{ where }E^{k}=V\text{ }^{k}-U\text{ }^{k}=\left(
e_{1}^{k},...,e_{N-1}^{k}\right) ^{T}
\end{equation*}

\begin{equation*}
\text{Table }10\text{ },\text{ }h=0.1
\end{equation*}

\begin{center}
\begin{tabular}{|c|c|}
\hline
$\ \ \ \ \ \ \ \ \ \ \ \ \ \ \ h_{t}$ & $\ 
\begin{tabular}{ccc}
$\ \ \ $\ $\ \ 10^{-3}$\ \ \ \ \ \ \  & \ $\ \ \ 10^{-5}$ \ \ \ \ \ \ \ \ \ 
& $10^{-7}$ \ \ \ \ 
\end{tabular}%
$ \\ \hline
$%
\begin{tabular}{c}
$\left\Vert E^{1}\right\Vert _{\infty }$ for
\end{tabular}%
\begin{tabular}{|c|}
\hline
$\alpha =1.2$ \\ \hline
$\alpha =1.4$ \\ \hline
$\alpha =1.6$ \\ \hline
$\alpha =1.8$ \\ \hline
$\alpha =1.9$ \\ \hline
\end{tabular}%
$ & \multicolumn{1}{|l|}{%
\begin{tabular}{|c|c|c|}
\hline
$\ \ \ 9.5736e\ast 10^{-4}$ & $1.3196\ast 10^{-6}\ $ & $5.2768\ast 10^{-9}\ $
\\ \hline
$\ 1.1294\ast 10^{-4}$ & $\ 1.2671\ast 10^{-7}$ & $2.0154\ast 10^{-10}$ \\ 
\hline
$2.3162\ast 10^{-5}$ & $1.2692\ast 10^{-8}$ & $7.9794\ast 10^{-12}$ \\ \hline
$1.53\ast 10^{-4}$ & $1.4449\ast 10^{-9}$ & $3.4062\ast 10^{-13}$ \\ \hline
$4.7963\ast 10^{-6}$ & $6.2306\ast 10^{-10}$ & $8.6153\ast 10^{-14}$ \\ 
\hline
\end{tabular}%
} \\ \hline
\end{tabular}
\end{center}

We see in the table 10, for the space step $h=0.1,$ and for the defferent
values of $\alpha ,$ the error norm tends to zeros when the time step $h_{t}$
takes values close to zeros, with convergence order $O(h+h_{t}).$

\section*{Conclusion}

In this paper, we study a problem with fractional derivatives with boundary
conditions of integral types. The study concerns a Caputo-type
advection-diffusion equation where the fractional order derivative $\alpha $
with respect to time with $1<\alpha <2$. The existence and uniqueness are
proven by the method of energy inequalities. The numerical study of this
problem based on the finite difference method. Applications on certain
examples clearly show that the numerical results obtained are very
satisfactory, where we see the approximate solution $u$ tends to the exact
solution $v$ for the defferent value of $\alpha .$

\bibliographystyle{aaai-named}
%\bibliography{acompat}

\end{document}